\documentclass[12pt]{article} % Remove a4paper for arxiv

\usepackage[margin=2cm]{geometry}
\usepackage{amsmath,amssymb,latexsym,amsthm}
\usepackage[active]{srcltx}

\theoremstyle{plain}
\newtheorem{proposition}{Proposition}[section]
\newtheorem{theorem}[proposition]{Theorem}
\newtheorem{corollary}[proposition]{Corollary}
\newtheorem{lemma}[proposition]{Lemma}
\theoremstyle{definition}
\newtheorem{definition}[proposition]{Definition}

\newtheorem{remark}[proposition]{Remark}

\newcommand{\ip}[2]{\langle #1,#2 \rangle}
\newcommand{\mc}{\mathcal}
\newcommand{\mf}{\mathfrak}
\newcommand{\G}{{\mathbb G}}
\newcommand{\C}{{\mathbb C}}
\newcommand{\vnten}{\overline\otimes}
\newcommand{\proten}{\widehat\otimes}
\newcommand{\op}{{\operatorname{op}}}
\newcommand{\Image}{{\operatorname{Im}}}

\begin{document}

\title{Completely bounded representations of convolution algebras
of locally compact quantum groups}
\author{Michael Brannan, Matthew Daws, Ebrahim Samei}
\maketitle

\begin{abstract}
Given a locally compact quantum group $\G$, we study the structure of 
completely bounded homomorphisms $\pi:L^1(\mathbb G)\rightarrow\mathcal B(H)$, and the question of when they are similar to $\ast$-homomorphisms. By analogy with the cocommutative case
(representations of the Fourier algebra $A(G)$), we are led to consider the
associated map $\pi^*:L^1_\sharp(\mathbb G) \rightarrow \mathcal B(H)$ given by $\pi^*(\omega) = \pi(\omega^\sharp)^*$.  We show that the corepresentation $V_\pi$ of $L^\infty(\mathbb G)$ associated to $\pi$ is invertible if and only if both $\pi$ and $\pi^*$ are completely bounded.  Moreover, we show that the co-efficient operators of such representations give rise to completely bounded multipliers of the dual convolution algebra $L^1(\hat \G)$.  An application of these results is that any (co)isometric corepresentation is automatically unitary. 
An averaging argument then shows that when $\mathbb G$ is amenable,
$\pi$ is similar to a $*$-homomorphism if and only if $\pi^*$ is completely bounded.  For compact Kac algebras, and for certain cases of $A(G)$, we show that any completely bounded homomorphism $\pi$ is similar to a $*$-homomorphism, without further assumption on $\pi^*$.  Using free product techniques, we construct new examples of compact quantum groups $\G$ such that $L^1(\G)$ admits bounded, but not completely bounded, representations.  


\noindent\emph{MSC classification (2010):} 20G42, 22D12, 22D15, 43A30,
   46L07, 46L51, 46L89, 81R50.
\end{abstract}

\section{Introduction} \label{sec:intro}

Given a locally compact group $G$, a \emph{uniformly bounded representation} of
$G$ on a Hilbert space $H$ is a weakly continuous homomorphism $\pi_0$ from $G$
into the invertible group of the algebra $\mc B(H)$ of bounded operators
on $H$, with $\|\pi_0\|:=\sup_{s\in G} \|\pi_0(s)\| < \infty$.  The study of
uniformly bounded representations of locally compact groups has been central
to the development of harmonic analysis and to understanding
the structure of operator algebras associated to locally compact groups
(see \cite{coha, dh, ps} for example).

A uniformly bounded representation $\pi_0:G \to \mc B(H)$ is called
unitarisable if there exists an invertible $T \in \mc B(H)$ such that
$T\pi_0(\cdot)T^{-1}$ is a unitary representation.  $G$ is called
\emph{unitarisable} if every uniformly bounded representation of
$G$ is unitarisable. In \cite{day,dix}, Day and Dixmier independently showed that
if $G$ is amenable, then $G$ unitarisable.  Perhaps one of the most celebrated
open problems concerning uniformly bounded representations is \emph{Dixmier's
similarity problem}, which asks the converse:  Is every unitarisable locally
compact group amenable?  For recent progress on this problem, see for
example \cite{mo, pis}.

In this paper we initiate the study of ``uniformly bounded
representations'' of locally compact quantum groups, and generalise
Dixmier's similarity problem to this context.
To motivate the objects of study in this paper, recall that for a locally compact
group $G$, there is a bijective correspondence between non-degenerate bounded representations
$\pi:L^1(G) \to \mc B(H)$, of the Banach $\ast$-algebra $L^1(G)$, and uniformly
bounded representations $\pi_0:G \to \mc B(H)$.  The correspondence between
$\pi$ and $\pi_0$ is given in the usual way by integration:
\begin{align*}
\big( \pi(\omega) \alpha | \beta \big)
= \int_{G} \omega(s)\big( \pi_0(s) \alpha | \beta \big) \ ds
\qquad (\omega \in L^1(G), \alpha, \beta \in H).
\end{align*}
Note that $\|\pi\| = \|\pi_0\|$.  Further, as $L^1(G)=L^\infty(G)_*$ is a
maximal operator space, $\pi$ is automatically completely bounded with
$\|\pi\|_{\mc {CB}(L^1(G),\mc B(H))} = \|\pi\|$.  Moreover, from the
above correspondence, it is easy to see that a non-degenerate representation
$\pi:L^1(G) \to \mc B(H)$ is (similar to) a $\ast$-representation of $L^1(G)$
if and only if $\pi_0$ is (similar to) a unitary representation.  

Thus, if $\G$ is a locally compact quantum group with von Neumann algebra
$L^\infty(\G)$ and convolution algebra $L^1(\G)$, we are led to consider
(completely) bounded representations $\pi:L^1(\G) \to \mc B(H)$. A representation
$\pi:L^1(\G) \to \mc B(H)$ is called a \emph{$\ast$-representation} if the
restriction of $\pi$ to the canonical dense $\ast$-subalgebra
$L^1_\sharp(\G) \subseteq L^1(\G)$ (with involution $\omega \mapsto \omega^\sharp$),
is a $\ast$-representation.  This is equivalent to the fact that
$\pi|_{L^1_\sharp(\G)} = \pi^*$, where $\pi^*:L^1_\sharp(\G) \to \mc B(H)$
is the representation obtained from $\pi$ via the formula
$\pi^*(\omega) = \pi(\omega^\sharp)^*$.  As usual, we say $\pi$ is
\emph{similar to a $\ast$-representation} if $T\pi(\cdot)T^{-1}$ is a
$\ast$-representation in the above sense, for some invertible $T \in \mc B(H)$.

A necessary condition for a representation $\pi:L^1(\G) \to \mc B(H)$ to be
similar to a $\ast$-representation is that both $\pi$ and $\pi^*$ extend to
\emph{completely bounded} maps from $L^1(\G)$ to
$\mc B(H)$ (see Proposition~\ref{prop:three}).  We therefore begin our investigation by focusing on properties of representations $\pi:L^1(\G) \to \mc B(H)$ where both $\pi$ and $\pi^*$ are completely bounded. 

Completely bounded representations $\pi:L^1(\G)\rightarrow\mc B(H)$
correspond to corepresentations $V_\pi\in L^\infty(\G)\vnten\mc B(H)$, that is,
operators with $(\Delta\otimes\iota)(V_\pi) = V_{\pi,13}V_{\pi,23}$.  We show
in Theorem~\ref{thm:inv_case} that when $V_\pi$ is invertible in
$L^\infty(\G)\vnten\mc B(H)$, then $\pi^*$ is automatically completely bounded.
By using the duality theory for locally compact quantum groups, and
the theory of completely bounded multipliers (\cite{dm,dm2,jnr}) we show in
Theorem~\ref{thm:one} that the converse is true-- if $\pi$ and $\pi^*$
are completely bounded, then $V_\pi$ is invertible (with an appropriate
modification when $\pi$ is a degenerate homomorphism).  Furthermore, invertible
corepresentations share many of the nice properties of unitary corepresentations--
they interact well with the unbounded antipode, and are ``strongly
continuous'', that is, $V_\pi \in M(C_0(\G) \otimes \mc B_0(H))$, see
Theorem~\ref{thm:summ1}.
We also show that the co-efficient operators of such
a representation $\pi$ naturally induce completely bounded multipliers on
the dual convolution algebra $L^1(\hat\G)$ (Corollary~\ref{cor:cbmults}).  This
generalises a classical result of de Canniere and Haagerup \cite[Theorem 2.2]{dh},
showing that coefficient functions of uniformly bounded representations of a group
$G$ are always completely bounded multipliers of the Fourier algebra $A(G)$.
We believe that invertible
corepresentations, equivalently, $\pi$ such that $\pi^*$ is also
completely bounded, should be thought of as the correct quantum generalisation
of a uniformly bounded group representation.

An interesting corollary of these results is that a corepresentation $U$
which is not necessarily unitary, but is at least a partial isometry, and
is either injective or surjective, must automatically be unitary.  This
follows since if $\pi$ is the associated homomorphism, then the condition on $U$
is enough to show that $\pi^*$ is completely bounded, and then $U$ must be
invertible.  An application of this is an improvement upon a construction of Kustermans:
the ``induced corepresentations'' constructed in \cite{kus2} are always
unitary, without any ``integrability'' condition; see Section~\ref{sec:coisocase}.

Another application of our results is to the similarity problem for amenable locally compact quantum groups.  When $\G$ is amenable, Theorem~\ref{thm:one} combined with some averaging techniques with invariant means shows that a representation $\pi:L^1(\G) \to \mc B(H)$ is similar to a $\ast$-representation if and only if both $\pi$ and $\pi^*$ are completely bounded maps (Theorem~\ref{thm:main}).  This generalises the results of Day and Dixmier on the unitarisability of amenable groups, as well as some known results for compact quantum groups.  When $\G$ is a compact Kac algebra (and for certain cases of $A(G)$) we are able to improve these results and show that every completely bounded representation $\pi:L^1(\G) \to \mc B(H)$ is similar to a $\ast$-representation, without assumption on $\pi^*$ (see Theorems~\ref{thm:Kac}, \ref{T:cc rep-amen group}, and \ref{T:cb rep-amen open SIN sub}).

In \cite{cs}, the first examples of bounded, but not completely bounded, representations of Fourier algebras were constructed.  Using non-commutative Khintchine inequalities from free probability theory, we show that similar constructions hold for certain free products of compact quantum groups; such representations cannot, of course, be similar to $*$-representations.  Moreover, for the representations constructed here, we show that there are co-efficient operators which fail to induce bounded multipliers of $L^1(\hat\G)$ (Corollary \ref{cor:freeproduct}).  This provides further evidence that \emph{completely} bounded maps are indeed the ``correct'' maps to study.

The paper is organised as follows.  In Section~\ref{sec:lcqg},
we review some basic facts about locally compact quantum groups which will be
needed in this paper, fixing some notation along the way.  We also make some
remarks on the extended Haagerup tensor product.  In Section~\ref{sec:cbreps}
we show the correspondence between completely bounded representations $\pi$
on $L^1(\G)$ and corepresentations $V_\pi$, and show that if $\pi$ is similar
to a $*$-representation, or $V_\pi$ is invertible, then $\pi^*$ is also
completely bounded.  In Section~\ref{sec:invreps} we prove the converse to
this result.  We first study the duality theory of locally compact quantum
groups, trying to find a quantum analogue of the function space $L^1(G)
\cap A(G)$.  We then show that coefficients of representations $\pi$, such
that both $\pi$ and $\pi^*$ are completely bounded, induce completely
bounded multipliers.  This result is then applied to prove that $V_\pi$ is
invertible.  In Section~\ref{sec:simmprob}, we consider the similarity
problem for locally compact quantum groups, focusing on the amenable case, and
refining our results for compact Kac algebras in Section~\ref{sec:specialcases}.
In Section~\ref{sec:importance_cb}, we construct bounded representations of
$L^1(\G)$ which are not completely bounded, where $\G$ is a certain free product of compact quantum groups.
Section~\ref{sec:fourieralgebra} deals only with representations of Fourier
algebras, and extends some of the results obtained in \cite{bs}.
Finally Appendix~\ref{sec:app1} looks at weak$^*$-closed operators, and
proves a technical result about analytic generators of $\sigma$-weakly
continuous, one-parameter groups.

Finally, a quick word on notation.  We use $\ip{\cdot}{\cdot}$ to denote the
bilinear pairing between a Banach space and its dual, and $(\cdot|\cdot)$ to
denote the inner product of a Hilbert space.  We use the symbol $\overline\otimes$
for the von Neumann algebraic tensor product; for tensor products of
Hilbert spaces, or the minimal tensor products of C$^*$-algebras,
we write simply $\otimes$.  We shall use standard results from the theory
of operator spaces, following \cite{ER} for example.

\medskip
\noindent\textbf{Acknowledgements:}
This paper was initiated at the Canadian Abstract Harmonic Analysis Symposium 2010,
and the first and second authors wish to thank the third author, Yemon Choi,
and the staff at the University of Saskatchewan for their hospitality.
We acknowledge the financial support of PIMS, NSERC and the University of
Saskatchewan.  The third author was supported by an NSERC Discovery Grant.

\section{Locally compact quantum groups} \label{sec:lcqg}

For the convenience of the reader, we give a brief overview of the theory of
locally compact quantum groups.  Our main references are \cite{kv} and \cite{kvvn}.
The first reference is a self-contained account of the C$^*$-algebraic approach
to locally compact quantum groups, and the second reference discusses the
von Neumann algebraic approach.   For other readable introductions, see
\cite{kus1} and \cite{vphd}. 

A \emph{Hopf-von Neumann algebra} is a pair $(M,\Delta)$ where $M$ is a
von Neumann algebra and $\Delta:M \to M\overline{\otimes}M$ is a unital
normal $\ast$-homomorphism which is coassociative:
$(\iota \otimes \Delta) \circ \Delta = (\Delta \otimes \iota) \circ \Delta$.
Then $\Delta_*$ induces a completely contractive Banach algebra product on the
predual $M_*$. We shall write the product in $M_*$ by juxtaposition, so
\[ \ip{x}{\omega\omega'} = \ip{\Delta(x)}{\omega \otimes \omega'}
\qquad (x \in M, \omega, \omega' \in M_*). \]

Recall the notion of a normal semi-finite faithful weight $\varphi:M^+ \to [0,\infty]$
(see \cite[Chapter VII]{tak2} for example). We let
\[ \mf n_\varphi = \{x \in M : \varphi(x^*x) < \infty \}, \ \ \
\mf m_\varphi = \operatorname{span}\{ x^*y : x, y \in \mf n_\varphi \},
\ \ \ \mf m^+_\varphi = \{x  \in M^+ : \varphi(x) < \infty \}. \]
Then $\mf m_\varphi$ is a hereditary $\ast$-subalgebra of $M$, $\mf n_\varphi$
is a left ideal, and $\mf m^+_\varphi$ is equal to $M^+ \cap \mf m_\varphi$. We
can perform the GNS construction for $\varphi$, which leads to a Hilbert space $H$,
a linear map $\Lambda: \mf n_\varphi \to H$ with dense range, and a unital normal
$\ast$-homomorphism $\pi : M \to \mc B(H)$ with $\pi(x)\Lambda(y) = \Lambda(xy)$.
In future, we shall tend to suppress $\pi$ in our notation.  Tomita-Takesaki theory
gives us the modular conjugation $J$ and the modular automorphism group
$(\sigma_t)_{t\in \mathbb R}$.  Recall that there is a (unbounded) positive
non-singular operator $\nabla$ (denoted $\Delta$ in \cite{tak2}) which induces
$(\sigma_t)$ via $\sigma_t(x) = \nabla^{it} x \nabla^{-it}$ for $x\in M$ and
$t\in\mathbb R$.  Finally, $M$ is in standard position on $H$, so in particular,
for each $\omega\in M_*$, there are $\xi,\eta\in H$ with $\omega=\omega_{\xi,\eta}$,
where $\ip{x}{\omega_{\xi,\eta}} = (x\xi|\eta)$ for $\xi,\eta\in H$.

A von Neumann algebraic \emph{locally compact quantum group} is a Hopf-von Neumann
algebra $(M,\Delta)$, together with faithful normal semifinite weights
$\varphi, \psi$ which are left and right invariant, respectively. This means that
\[ \varphi\big((\omega \otimes \iota)\Delta(x)\big) = \varphi(x) \ip{1}{\omega}, \quad
\psi\big((\iota \otimes \omega)\Delta(y)\big) = \psi(y)\ip{1}{\omega}
\qquad (\omega \in M^+_*, x \in \mf m_\varphi^+, y \in \mf m_\psi^+). \]
Using these weights, we can construct an antipode $S$, which will in general be
unbounded. Then $S$ has a decomposition $S = R\tau_{-i/2}$, where $R$ is the unitary
antipode, and $(\tau_t)_{ t \in \mathbb R}$ is the scaling group. The unitary
antipode $R$ is a normal $\ast$-antiautomorphism of $M$, and
$\Delta R = \sigma(R \otimes R)\Delta$, where
$\sigma: M\overline{\otimes} M \to M\overline{\otimes}M$ is the tensor swap map.
As $R$ is normal, it drops to an isometric linear map $R_*: M_* \to M_*$, which is
anti-multiplicative. As usual, we make the canonical choice that
$\varphi = \psi \circ R$.

Associated to $(M,\Delta)$ is a \emph{reduced} C$^\ast$-algebraic quantum group
$(A,\Delta)$. Here $A$ is a C$^*$-subalgebra of $M$, and
$\Delta: A \to M(A \otimes A)$, the multiplier algebra of $A\otimes A$.  Here we
identify $M(A \otimes A)$ with a C$^\ast$-subalgebra of $M \overline{\otimes}M$;
indeed, given a C$^*$-algebra $B\subseteq\mc B(H)$, we can always identify the
multiplier algebra $M(B)$ with $\{ x\in B'' : xb,bx\in B \ (b\in B) \}$.

There is a unitary $W$, the \emph{fundamental unitary}, acting
on $H \otimes H$ (the Hilbert space tensor product) such that
$\Delta(x) = W^*(1\otimes x)W$ for $x \in M$.  Then $W$ is multiplicative,
in the sense that $W_{12} W_{13} W_{23} = W_{23} W_{12}$.  Here, and later,
we use the standard leg number notation: $W_{12} = W\otimes I, W_{13}
= \Sigma_{23} W_{12} \Sigma_{23}$ and so forth, where
here $\Sigma: H \otimes H \to H \otimes H$ is the swap map.
The \emph{left-regular representation}
of $M_*$ is the map $\lambda : \omega \mapsto (\omega \otimes \iota )W$.  This
is an injective homomorphism, the norm closure of $\lambda(M_*)$ is a
C$^\ast$-algebra denoted by $\hat A$, and the $\sigma$-weak closure of
$\lambda(M_*)$ is a von Neumann algebra, which we denote by $\hat{M}$.  We define a
coproduct $\hat{\Delta}:\hat M \to \hat M \overline{\otimes}\hat M$ by
$\hat\Delta(x) = \hat W^*(1\otimes x) \hat W$ , where $\hat W = \Sigma W^∗\Sigma$.
Then we can find
invariant weights to turn $(\hat M , \hat \Delta)$ into a locally compact quantum
group -- the dual quantum group to $M$.  We then have the Pontryagin duality theorem
which states that $\hat{\hat{M}}= M$ canonically.
There is a nice link between the scaling group of $M$, and the modular theory of
the (left) Haar weight of $\hat M$, in that
$\tau_t(x) = \hat\nabla^{it} x \hat\nabla^{-it}$ for $x\in M, t\in\mathbb R$.

In this paper we will use the symbol $\G$ to indicate the abstract ``object'' to be
thought of as a locally compact quantum group.  Inspired by the classical examples
of locally compact groups, we write $L^\infty(\G)$ for $M$,
$C_0(\G)$ for $A$, $C^b(\G)$ for $M(A)$ (the multiplier algebra of $A$),
$L^1(\G)$ for $M_*$, and $L^2(\G)$ for $H$.  We denote the locally compact quantum
group dual to $\G$ by $\hat \G$, and use the notations $L^\infty(\hat \G) =
\hat M$, $L^1(\hat \G) = \hat M_*$, and so on.  In order to distinguish between
objects attached to $ \hat \G$ from those attached to $\G$, we decorate elements
in $L^\infty(\hat \G)$, $L^1(\hat \G)$ etc. with ``hats''. For example, we write
$\hat x \in L^\infty(\hat\G)$, and $\hat \omega \in L^1(\hat\G)$.

Since the antipode $S$ of $\G$ is unbounded in general, there is not a natural
way to turn $L^1(\G)$ into a $*$-algebra.  However, $L^1(\G)$ always contains a
dense subalgebra which has a conjugate-linear involution.  We follow \cite{kus},
see also \cite[Section~2]{kvvn}.  Define $L^1_\sharp(\G)$ to be the collection of
those $\omega\in L^1(\G)$ such that there exists $\omega^\sharp \in L^1(\G)$ with
$\ip{x}{\omega^\sharp} = \overline{ \ip{S(x)^*}{\omega} }$ for each $x\in D(S)$.
Then $L^1_\sharp(\G)$ is a dense subalgebra of $L^1(\G)$, and
$\omega\mapsto\omega^\sharp$ defines an involution on $L^1_\sharp(\G)$.  As
\cite[Proposition~3.1]{kus} shows, given $\omega\in L^1(\G)$, we have that
$\omega\in L^1_\sharp(\G)$ if and only if $\lambda(\omega)^* = \lambda(\omega')$
for some $\omega'\in L^1(\G)$.  Furthermore, the left regular representation
$\lambda$ becomes a $*$-homomorphism, when restricted to $L^1_\sharp(\G)$.

When $G$ is a locally compact group, we write $VN(G)$ for $L^\infty(\hat G)$
(the von Neumann algebra generated by the left regular representation of
$\lambda_G:G \to \mc U(L^2(G))$), and write $A(G)$ for $L^1(\hat G)$, and call
this the Fourier algebra of $G$.  The Fourier algebra was first studied as a
Banach algebra by Eymard \cite{eymard}.  The left regular representation
$\hat\lambda:A(G) \to C_0(G)$, is a completely contractive $\ast$-homomorphism,
which allows us to identify $A(G)$ with a dense $\ast$-subalgebra of $C_0(G)$,
the identification being given by
\[ \hat\lambda(\hat\omega)(t) = \ip{\lambda_G(t^{-1})}{\hat\omega}
\qquad (\hat \omega \in A(G), t \in G).\]  
See \cite{tt2} for details (note that the above identification differs
from the one considered by Eymard, where $t$ is replaced by $t^{-1}$ in
\cite{eymard}).  So as to not overburden notation, in many places 
(especially in Section~\ref{sec:fourieralgebra}) we will drop
the notation $\hat \lambda$ from above, and just view $A(G)$ as a concrete
$\ast$-subalgebra of $C_0(G)$.

\subsection{The extended Haagerup tensor product}\label{sec:hatp}

Let $M$ and $N$ be von Neumann algebras.
The \emph{extended (or weak$^*$) Haagerup tensor product} of $M$ with $N$
is the collection of $x\in M\vnten N$ such that
\[ x = \sum_{i} x_i \otimes y_i \]
with $\sigma$-weak convergence, and such that $\sum_{i} x_i x_i^*$ and
$\sum_{i} y_i^* y_i$ converge $\sigma$-weakly in $M$ and $N$ respectively.
The natural norm is then
\[ \|x\|_{eh} = \inf\Big\{ \Big\|\sum_{i} x_i x_i^*\Big\|^{1/2}
\Big\|\sum_{i} y_i^* y_i\Big\|^{1/2} :
x = \sum_{i} x_i\otimes y_i \Big\}, \]
and we write $M \otimes_{eh} N$ for the resulting normed space.
See \cite{bsm} or \cite{er2} for further details.

Suppose we are given $(x_i)_{i \in I}\subseteq M$ and
$(y_i)_{i \in I}\subseteq N$ with $\sum_{i} x_i x_i^*$ and
$\sum_{i } y_i^* y_i$ converging $\sigma$-weakly.  Let $M\subseteq\mc B(H)$
and $N\subseteq\mc B(K)$.  Then there are bounded
maps $X:H\rightarrow H\otimes\ell^2(I)$ and $Y:K\rightarrow K\otimes\ell^2(I)$
given by
\[ X(\xi) = \sum_{i } x_i^*(\xi) \otimes \delta_i, \qquad
Y(\xi) = \sum_{i} y_i(\xi) \otimes \delta_i\qquad (\xi\in L^2(\G)), \]
where $(\delta_i)_{i \in I}$ is the canonical orthonormal basis for $\ell^2(I)$.
Indeed, $X^*X = \sum_i x_i x_i^*$ and $Y^*Y = \sum_i y_i^*y_i$.  Letting
$\Sigma:K\otimes\ell^2(I) \rightarrow \ell^2(I)\otimes K$ be the swap map,
a simple calculation shows that
\[ (X^*\otimes 1)(1\otimes\Sigma)(1\otimes Y) = \sum_i x_i \otimes y_i
\in M\vnten N \subseteq\mc B(H\otimes K). \]
This argument hence shows that $\sum_i x_i\otimes y_i$ always converges
$\sigma$-weakly in $M\vnten N$, as long as $\sum_{i \in I} x_i x_i^*$ and
$\sum_{i \in I} y_i^* y_i$ converge $\sigma$-weakly.
Finally, note that if $H=K$ then
\[ \sum_i x_i y_i = X^*Y \in\mc B(H), \]
and in particular, if $M=N$, then $\sum_i x_iy_i$ converges $\sigma$-weakly in
$M$ to $X^*Y$.

\section{Completely bounded homomorphisms and corepresentations}
\label{sec:cbreps}

In this section, we investigate the correspondence between completely
bounded representations $\pi$ of $L^1(\G)$, and corepresentations of $U \in
L^\infty(\G)\vnten\mc B(H)$ (which are not assumed unitary).  We show how
various naturally defined variants of $\pi$ are associated to variants of $U$.
In particular, we show that $\pi$ is similar to a $*$-representation if
and only if $U$ is similar (in the sense of corepresentations) to a unitary.
One of our variants of $\pi$ is defined using the unbounded $*$-map on
$L^1_\sharp(\G)$.  We show that if $U$ is invertible, then all of the
variants of $\pi$ are completely bounded.  A principle result,
Theorem~\ref{thm:main2} below, is that the converse to this
result is also true; and hence we have a bijection between invertible (not
necessarily unitary) corepresentations and a certain class of completely
bounded representations.

We are interested in completely bounded homomorphisms $\pi:L^1(\G)
\rightarrow \mc B(H)$.  We have the standard identifications
(see \cite[Chapter~7]{ER}):
\[ \mc{CB}(L^1(\G),\mc B(H)) \cong \big( L^1(\G) \proten \mc T(H) \big)^*
\cong L^\infty(\G) \vnten \mc B(H). \]
Here $\proten$ denotes the operator space projective tensor product, and
$\vnten$ the von Neumann algebraic tensor product.  We use the notation
that $\pi:L^1(\G)\rightarrow\mc B(H)$ is associated with $V_\pi
\in L^\infty(\G)\vnten\mc B(H)$, the relationship is that
\[ \pi(\omega) = (\omega\otimes\iota)V_\pi \qquad (\omega\in L^1(\G)). \]
We use the standard leg numbering notation, see for example
\cite{bs2, tt}.

\begin{lemma}
Let $\pi:L^1(\G)\rightarrow\mc B(H)$ be a completely bounded map, associated
to $V_\pi\in L^\infty(\G)\vnten\mc B(H)$.  Then $\pi$ is a homomorphism if and only
if $(\Delta\otimes\iota)V_\pi = V_{\pi,13} V_{\pi,23}$.
Similarly, $\pi$ is an anti-homomorphism if and only if
$(\Delta\otimes\iota)V_\pi = V_{\pi,23} V_{\pi,13}$.
\end{lemma}
\begin{proof}
For $\omega_1,\omega_2\in L^1(\G)$,
\[ \pi(\omega_1\omega_2)
= \big( \Delta_*(\omega_1\otimes\omega_2) \otimes\iota\big) V_\pi
= (\omega_1\otimes\omega_2\otimes\iota) \big((\Delta\otimes\iota)V_\pi\big), \]
while
\[ \pi(\omega_1) \pi(\omega_2)
= (\omega_1\otimes\iota)V_\pi (\omega_2\otimes\iota) V_\pi
= (\omega_1\otimes\omega_2\otimes\iota) \big( V_{\pi,13} V_{\pi,23} \big), \]
and so $\pi$ is a homomorphism if and only if
$(\Delta\otimes\iota)V_\pi = V_{\pi,13} V_{\pi,23}$.  The result for
anti-homomorphisms is analogous.
\end{proof}

The condition that $(\Delta\otimes\iota)V_\pi = V_{\pi,13} V_{\pi,23}$
says that $V_\pi$ is a \emph{co-representation} of $L^\infty(\G)$.
Similarly, we say that $V_\pi$ is a \emph{co-anti-representation}
when $(\Delta\otimes\iota)V_\pi = V_{\pi,23} V_{\pi,13}$.

Following \cite{bs}, given a locally compact group $G$ and a representation
$\pi:A(G)\rightarrow\mc B(H)$, we define $\check\pi, \pi^*$ and $\tilde\pi$
by \[ \check\pi(\omega) = \pi(\check\omega), \quad
\pi^*(\omega) = \pi(\overline\omega)^*, \quad
\tilde\pi(\omega) = \check\pi(\overline\omega)^*
= \pi(\overline{\check\omega})^* \qquad (\omega\in A(G)),  \]
where $\check \omega (g) = \omega(g^{-1})$ (viewing $A(G)$ as a subalgebra of $ C_0(G)$).

For $\omega\in L^1(\G)$, define $\omega^*$ by $\ip{x}{\omega^*} =
\overline{ \ip{x^*}{\omega} }$ for $x\in L^\infty(\G)$.  As $\Delta$ is a
$*$-homomorphism, it is easy to see that $L^1(\G)\rightarrow L^1(\G); \omega
\mapsto \omega^*$ is a conjugate linear homomorphism.  Recalling the definition
of $L^1_{\sharp}(\G)$ from Section~\ref{sec:lcqg}, 
given $\pi:L^1(\G) \rightarrow\mc B(H)$, we define
\begin{align*}
\check\pi&: L^1_\sharp(\G)^* \rightarrow \mc B(H);\quad
\omega\mapsto \pi\big( (\omega^*)^\sharp \big), \\
\pi^*&: L^1_\sharp(\G) \rightarrow \mc B(H);\quad
\omega\mapsto \pi(\omega^\sharp)^*, \\
\tilde\pi&: L^1(\G)\rightarrow\mc B(H);\quad
\omega\mapsto \pi(\omega^*)^*.
\end{align*}
These are defined by analogy with the $A(G)$ case, and the following results
show that they are natural choices.

\begin{proposition}\label{prop:two}
Let $\pi:L^1(\G)\rightarrow\mc B(H)$ be a bounded map.  Then $\pi$ is
completely bounded if and only if $\tilde\pi$ is completely bounded,
and in this case, $V_\pi^* = V_{\tilde\pi}$.  Furthermore, $\check\pi$
extends to a completely bounded map defined on $L^1(\G)$ if and only if
$\pi^*$ extends to a completely bounded map defined on $L^1(\G)$, and in
this case, $V_{\check\pi}^* = V_{\pi^*}$.
\end{proposition}
\begin{proof}
Let $U\in L^\infty(\G)\vnten\mc B(H)$ and define $\phi:L^1(\G)\rightarrow
\mc B(H)$ by $\phi(\omega)=(\omega\otimes\iota)(U)$.  Then
\[ \phi(\omega^*)^* = \big( (\omega^*\otimes\iota)(U) \big)^*
= \big( (\omega\otimes\iota)(U^*)^* \big)^* = (\omega\otimes\iota)(U^*). \]
It follows immediately that $\pi$ is completely bounded if and only if
$\tilde\pi$ is, and that $V_\pi^* = V_{\tilde\pi}$.  As $\pi^*(\omega^*)
= \check\pi(\omega)^*$, similarly $\check\pi$ extends to completely bounded
map if and only if $\pi^*$ does, and $V_{\check\pi}^* = V_{\pi^*}$.
\end{proof}

We wish to show that $*$-representations of $L^1_\sharp(\G)$ do give rise
to completely bounded $\pi$ such that also $\pi^*$ is completely bounded.
To do this, we need to investigate how unitary, or more generally, invertible
corepresentations interact with the antipode.  

Let us recall a
characterisation of the antipode using ideas very close to the extended Haagerup
tensor product.  This is shown in \cite[Corollary~5.34]{kv} in the C$^*$-algebra
setting, but the same proofs work in the von Neumann case.  Let $x,y\in L^\infty(\G)$
be such that there are families $(x_i),(y_i)\subseteq L^\infty(\G)$ with
$\sum_i x_i x_i^* < \infty$ and $\sum y_i^*y_i<\infty$, and such that
\[ x\otimes 1 = \sum_i \Delta(x_i)(1\otimes y_i), \qquad
y\otimes 1 = \sum_i (1\otimes x_i)\Delta(y_i). \]
Then $x\in D(S)$ and $S(x)=y$.  Note that these sums converge $\sigma$-weakly,
compare Section~\ref{sec:hatp}.  For further details on characterising the
antipode in this way, see \cite[Proposition~5.6]{vv}.

\begin{theorem}\label{thm:inv_case}
Let $V_\pi \in L^\infty(\G) \vnten \mc B(H)$ be an invertible corepresentation,
associated to a completely bounded homomorphism $\pi:L^1(\G)\rightarrow\mc B(H)$.
Then $\check\pi$ is a completely bounded anti-homomorphism, and $V_\pi^{-1}
= V_{\check\pi}$.
\end{theorem}
\begin{proof}
Let $\alpha,\beta\in H$, let $(f_i)$ be an orthonormal basis of $H$, and set
\[ x_i = (\iota\otimes \omega_{f_i,\beta})(V_\pi), \qquad
y_i = (\iota\otimes \omega_{\alpha,f_i})(V_\pi^{-1}). \]
As in Section~\ref{sec:hatp}, it is easy to see that
$\sum_i x_i x_i^* = (\iota\otimes\omega_\alpha)(V_\pi V_\pi^*)$
and $\sum_i y_i^* y_i = (\iota\otimes\omega_\beta)((V_\pi^{-1})^*V_\pi^{-1})$,
with the sums converging $\sigma$-weakly.
Now, using that $(\Delta\otimes\iota)(V_\pi) = V_{\pi,13} V_{\pi,23}$,
\begin{align*}
\sum_i \Delta(x_i)(1\otimes y_i) &=
\sum_i (\iota\otimes\iota\otimes\omega_{f_i,\beta})((\Delta\otimes\iota)(V_\pi))
   (1\otimes y_i) \\
&= \sum_i (\iota\otimes\iota\otimes\omega_{f_i,\beta})(V_{\pi,13} V_{\pi,23})
   (1\otimes (\iota\otimes\omega_{\alpha,f_i})(V_\pi^{-1})).
\end{align*}
Now, for any Hilbert space $K$, and $S,T\in\mc B(K\otimes H)$,
\[ \sum_i (\iota\otimes\omega_{f_i,\beta})(S) (\iota\otimes\omega_{\alpha,f_i})(T)
= (\iota\otimes\omega_{\alpha,\beta})(ST). \]
%This follows, as for $\xi,\eta\in K$, and $(e_i)$ an orthonormal basis of $K$,
%\begin{align*}
%\sum_i \big( (\iota\otimes\omega_{f_i,\beta})(S) (\iota\otimes\omega_{\alpha,f_i})(T)
%   \xi \big| \eta \big) &=
%\sum_{i,j} \big( (\iota\otimes\omega_{f_i,\beta})(S) e_j \big| \eta \big)
%   \big( (\iota\otimes\omega_{\alpha,f_i})(T) \xi \big| e_j \big) \\
%&= \sum_{i,j} \big( S(e_j\otimes f_i) \big| \eta\otimes\beta \big)
%   \big( T(\xi\otimes\alpha) \big| e_j\otimes f_i \big) \\
%&= \big( ST(\xi\otimes\alpha)\big| \eta\otimes\beta \big)
%= \big( (\iota\otimes\omega_{\alpha,\beta})(ST) \xi \big| \eta \big).
%\end{align*}
Thus
\[ \sum_i \Delta(x_i)(1\otimes y_i)
= (\iota\otimes\iota\otimes\omega_{\alpha,\beta})(V_{\pi,13} V_{\pi,23} V_{\pi,23}^{-1})
= (\iota\otimes\omega_{\alpha,\beta})(V_\pi) \otimes 1. \]

Similarly, as $\Delta$ is a homomorphism, $(\Delta \otimes \iota)(V_\pi^{-1}) 
= V_{\pi,23}^{-1} V_{\pi,13}^{-1}$, and so
\begin{align*} \sum_i (1\otimes x_i)\Delta(y_i)
&= \sum_i (1\otimes(\iota\otimes\omega_{f_i,\beta})(V_\pi))
   (\iota\otimes\iota\otimes\omega_{\alpha,f_i})(V_{\pi,23}^{-1} V_{\pi,13}^{-1}) \\
&= (\iota\otimes\iota\otimes\omega_{\alpha,\beta})(V_{\pi,23}V_{\pi,23}^{-1}
   V_{\pi,13}^{-1})
= (\iota\otimes\omega_{\alpha,\beta})(V_\pi^{-1}) \otimes 1.
\end{align*}
Thus $(\iota\otimes\omega_{\alpha,\beta})(V_\pi) \in D(S)$ and 
$S\big( (\iota\otimes\omega_{\alpha,\beta})(V_\pi) \big)
=  (\iota\otimes\omega_{\alpha,\beta})(V_\pi^{-1})$.

Let $\omega^*\in L^1_\sharp(\G)$, so we have that
\[ \ip{x}{(\omega^*)^\sharp} = \overline{\ip{S(x)^*}{\omega^*}}
= \ip{S(x)}{\omega} \qquad (x\in D(S)). \]
It follows that
\begin{align*} ( \check\pi(\omega) \alpha | \beta )
&= \ip{ (\iota\otimes\omega_{\alpha,\beta}(V_\pi) }{(\omega^*)^\sharp}
= \ip{ S((\iota\otimes\omega_{\alpha,\beta}(V_\pi)) }{\omega} \\
&= \ip{ (\iota\otimes\omega_{\alpha,\beta})(V_\pi^{-1}) }{\omega}
= \ip{V_\pi^{-1}}{\omega\otimes\omega_{\alpha,\beta}}. \end{align*}
As such $\omega$ are dense in $L^1(\G)$, it follows that $\check\pi$ is
completely bounded, with $V_{\check\pi} = V_\pi^{-1}$, as claimed.
\end{proof}

We now show that any $*$-homomorphism on $L^1_\sharp(\G)$ gives rise to
a completely bounded homomorphism $\pi$ such that also $\pi^*$ is
completely bounded.  We use a result of Kustermans that any $*$-representation
has a ``generator''-- that is, $V_\pi$ exists and is unitary.  The interaction
between a unitary corepresentation and the antipode (which, in some sense,
we generalised in the previous proposition) is of course well-known, see
\cite[Theorem~1.6]{woro} for example.

\begin{proposition}\label{prop:three}
Let $\pi:L^1_\sharp(\G)\rightarrow\mc B(H)$ be a homomorphism which is similar
to a $*$-homomorphism.  Then $\pi$ and $\pi^*$ extend by continuity to
completely bounded homomorphisms $L^1(\G)\rightarrow\mc B(H)$, and 
$\check\pi$ and $\tilde\pi$ extend by continuity to completely bounded
anti-homomorphisms $L^1(\G)\rightarrow\mc B(H)$.
\end{proposition}
\begin{proof}
We first show this in the case when $\pi$ is a $*$-homomorphism.
Then $\pi(\omega^\sharp) = \pi(\omega)^*$ for each $\omega\in L^1_\sharp(\G)$,
and so $\pi^* = \pi$.  As $\pi$ is a $*$-homomorphism, it is necessarily
contractive (see \cite[Chapter~I, Proposition~5.2]{tak} for example) and so
extends by continuity to a homomorphism defined on all of $L^1(\G)$.

To show that $\pi$ extends to a completely bounded homomorphism, we use
a non-trivial result of Kustermans.  Suppose for the moment that $\pi:
L^1_\sharp(\G)\rightarrow\mc B(H)$ is non-degenerate.  By \cite[Corollary~4.3]{kus},
there exists a unitary $U\in L^\infty(\G) \vnten \mc B(H)$
(actually, $U\in M(C_0(\G) \otimes \mc B_0(H)) \subseteq
L^\infty(\G) \vnten \mc B(H)$)
such that $\pi(\omega) = (\omega\otimes\iota)U$ for each $\omega\in L^1_\sharp(\G)$.
It follows immediately that $\pi$ does indeed extend to a completely bounded
homomorphism from $L^1(\G)$, and that thus actually $U = V_\pi$.  As $U$ is
unitary, so invertible, the previous propositions shows that $\pi^*$ is
also completely bounded.

If $\pi$ is degenerate, then as $\pi$ is a $*$-homomorphism, we can orthogonally
decompose $H$ as $H_1 \oplus H_2$, where $\pi$ restricts to $H_1$, and $\pi(\omega)\xi=0$
for each $\xi\in H_2,\omega\in L^1_\sharp(\G)$.  Let $\pi_1$ be the restriction of
$\pi$ to $H_1$; so $\pi_1:L^1(\G)\rightarrow\mc B(H_1)$ is a completely bounded
homomorphism.  It follows immediately that the same must be true of $\pi$
and of $\pi^*$.

If $\pi$ is only similar to a $*$-representation, then there exists an invertible
$T\in\mc B(H)$ such that $\theta:L^1_\sharp(\G)\rightarrow\mc B(H),
\omega\mapsto T^{-1}\pi(\omega)T$ is a $*$-homomorphism.  Then $\theta$ extends to
a completely bounded homomorphism, and hence so also does $\pi = T\theta(\cdot)T^{-1}$.
The same result holds for $\pi^*$, as
\[ \pi^*(\omega) = \pi(\omega^\sharp)^* = (T^{-1})^* \theta^*(\omega) T^*
\qquad (\omega\in L^1_\sharp(\G)). \]

Having established that $\pi$ and $\pi^*$ extend to completely bounded maps
$L^1(\G)\rightarrow\mc B(H)$, it follows from Proposition~\ref{prop:two}
that also $\check\pi$ and $\tilde\pi$ extend to completely bounded maps
$L^1(\G)\rightarrow\mc B(H)$ which are easily seen to be anti-homomorphisms.
\end{proof}

\section{Invertible corepresentations}\label{sec:invreps}

In the previous section, we showed that a necessary condition for
$\pi:L^1(\G)\rightarrow\mc B(H)$ to be similar to a $*$-representation is
that both $\pi$ and $\pi^*$ are completely bounded.  Furthermore, if
$\pi$ is associated to a corepresentation $V_\pi$, and $V_\pi$ is
invertible, then $\pi^*$ is completely bounded.  In this section, we use the
duality theory of locally compact quantum groups to show the converse:
if $\pi$ and $\pi^*$ are completely bounded, then $V_\pi$ is invertible
(together with an appropriate interpretation of this when $\pi$ is a
degenerate homomorphism).

\subsection{Identification of ``function'' spaces}

In \cite{bs}, an important component of the argument was to use the space
$L^1(G) \cap A(G)$, identified as a function space on $G$.  For a quantum
group, we would like to be able to talk about the space
$L^1(\G) \cap L^1(\hat\G)$; to make sense of this, we shall embed dense
subspaces of $L^1(\G)$ and $L^1(\hat\G)$ into $L^2(\G)$, making use of the
fact that we always identify $L^2(\G)$ with $L^2(\hat\G)$.  This section
is slightly technical-- the key result is Proposition~\ref{prop:seven} below,
which is used in the following section.

We first introduce the non-standard, but instructive notation
\[ L^1(\G) \cap L^2(\G) = \big\{ \omega\in L^1(\G) : \exists\, \xi\in L^2(\G),
\ip{x^*}{\omega} = (\xi|\Lambda(x)) \ (x\in \mf n_\varphi) \big\}. \]
This space is denoted by $\mc I$ in \cite[Section~1.1]{kvvn} and
\cite[Notation~8.4]{kv}.  Then $\hat\varphi$
is the weight with GNS construction $(L^2(\G),\iota,\hat\Lambda)$ where
$\lambda(L^1(\G) \cap L^2(\G))$ forms a $\sigma$-strong$^*$ core for $\hat\Lambda$,
and we have that
\[ \ip{x^*}{\omega} = \big( \hat\Lambda(\lambda(\omega)) \big| \Lambda(x) \big)
\qquad (x\in \mf n_\varphi, \omega\in L^1(\G) \cap L^2(\G)). \]
Then \cite[Result~8.6]{kv} shows that $L^1(\G)\cap L^2(\G)$ is a left
ideal in $L^1(\G)$ and
\[ \hat\Lambda\big( \lambda(\omega\omega_1) \big)
= \lambda(\omega) \hat\Lambda\big( \lambda(\omega_1) \big)
\qquad (\omega\in L^1(\G), \omega_1\in L^1(\G)\cap L^2(\G)). \]
This formula is of course immediate from the fact that $\hat\Lambda$ is a GNS
map, but this reasoning is circular(!) if one is following \cite[Section~8]{kv}.
Further, \cite[Result~8.6]{kv} can easily be adapted (or just perform the obvious
calculation) to show that $L^1(\G)\cap L^2(\G)$ is a left $L^\infty(\G)$ module, and
\[ \hat\Lambda\big( \lambda(x\omega_1) \big)
= x \hat\Lambda\big( \lambda(\omega_1) \big)
\qquad (x\in L^\infty(\G), \omega_1\in L^1(\G)\cap L^2(\G)). \]

Continuing to follow \cite[Section~8]{kv}, we find that
there is a norm continuous, one-parameter group $(\rho_t)_{t\in\mathbb R}$
of isometries on $L^1(\G)$, such that $\rho_t$ is an algebra homomorphism for
each $t$, and with $\hat\sigma_t(\lambda(\omega)) = \lambda(\rho_t(\omega))$ for
$\omega\in L^1(\G)$ (where $(\hat \sigma_t)_t$ is the modular automorphism group
associated to the dual left Haar weight $\hat \varphi$).  As observed before
\cite[Proposition~2.8]{kvvn}, each $\rho_t$ maps $L^1_\sharp(\G)$ into itself, and
$\rho_t(\omega^\sharp) = \rho_t(\omega)^\sharp$ for $\omega\in L^1_\sharp(\G)$.
Finally, for $\omega\in L^1(\G)\cap L^2(\G)$, also
$\rho_t(\omega) \in L^1(\G)\cap L^2(\G)$ and
$\hat\Lambda(\lambda(\rho_t(\omega))) =
\hat\nabla^{it} \hat\Lambda(\lambda(\omega))$.

\begin{lemma}\label{lem:one}
The collection of $\omega\in L^1(\G)\cap L^2(\G)$ with $\omega\in D(\rho_i)$,
$\rho_i(\omega)\in L^1_\sharp(\G)$, and $\rho_i(\omega)^\sharp \in L^1(\G)
\cap L^2(\G)$ is dense in $L^1(\G)$.  Furthermore, the resulting collection
of vectors $\hat\Lambda(\lambda(\omega))$ is dense in $L^2(\G)$.
\end{lemma}
\begin{proof}
We use a ``smearing argument'', compare (for example) \cite[Proposition~5.21]{kus1}.
Let $\omega\in L^1_\sharp(\G) \cap (L^1(\G)\cap L^2(\G))$.  Such $\omega$ are
dense in $L^1(\G)$ by \cite[Lemma~2.5]{kvvn}.
For $n\in\mathbb N$ and $z\in\mathbb C$ define
\[ \omega(n,z) = \frac{n}{\sqrt\pi} \int_{\mathbb R} \exp(-n^2(t+z)^2) \rho_t(\omega)
\ dt. \]
Then $\omega(n,0)$ is analytic for $\rho$ and $\rho_z(\omega(n,0)) = \omega(n,z)$.
As $\rho$ is norm continuous, it follows that $\omega(n,0)\rightarrow\omega$
in norm as $n\rightarrow\infty$.

For $x\in D(S)$, using that $\rho_t(\omega)^\sharp = \rho_t(\omega^\sharp)$,
\begin{align*} \ip{S(x)^*}{\omega(n,z)}
&= \frac{n}{\sqrt\pi} \int_{\mathbb R} \exp(-n^2(t+z)^2)
   \ip{S(x)^*}{\rho_t(\omega)} \ dt \\
&= \frac{n}{\sqrt\pi} \int_{\mathbb R} \exp(-n^2(t+z)^2)
   \overline{\ip{x}{\rho_t(\omega^\sharp)}} \ dt 
= \overline{ \ip{x}{(\omega^\sharp)(n,\overline z)} }.
\end{align*}
Hence $\omega(n,z)\in L^1_\sharp(\G)$ and $\omega(n,z)^\sharp
= (\omega^\sharp)(n,\overline z)$.

For $x\in\mf n_\varphi$,
\begin{align*} \ip{x^*}{(\omega^\sharp)(n,z)}
&= \frac{n}{\sqrt\pi}\int_{\mathbb R} \exp(-n^2(t+z)^2)
   \ip{x^*}{\rho_t(\omega^\sharp)} \ dt \\
&= \frac{n}{\sqrt\pi}\int_{\mathbb R} \exp(-n^2(t+z)^2)
   \big(\hat\Lambda(\lambda(\rho_t(\omega^\sharp)))\big|\Lambda(x)\big) \ dt
= \big( \xi \big| \Lambda(x) \big),
\end{align*}
where
\[ \xi = \frac{n}{\sqrt\pi}\int_{\mathbb R} \exp(-n^2(t+z)^2)
\hat\nabla^{it} \hat\Lambda(\lambda(\omega^\sharp)) \ dt. \]
In particular, $(\omega^\sharp)(n,z) \in L^1(\G) \cap L^2(\G)$.

It follows that $\{ \omega(n,0) : n\in\mathbb N, \omega\in L^1_\sharp(\G)\cap L^2(\G) \}$
is dense in $L^1(\G)$, analytic for $\rho$, satisfies that $\rho_z(\omega(n,0))
\in L^1_\sharp(\G)$ for all $z$, and satisfies that $\rho_z(\omega(n,0))^\sharp
\in L^1(\G)\cap L^2(\G)$.  Furthermore, a similar calculation to that above
yields that
\[ \hat\Lambda(\lambda(\omega(n,0)))
= \frac{n}{\sqrt\pi} \int_{\mathbb R} \exp(-n^2t^2) \hat\nabla^{it}
\hat\Lambda(\lambda(\omega)) \ dt. \]
This converges to $\hat\Lambda(\lambda(\omega))$ in norm, as
$n\rightarrow\infty$.  In particular, as $\omega$ varies, the
collection of vectors $\hat\Lambda(\lambda(\omega(n,0)))$ is dense in $L^2(\G)$.
\end{proof}

For the following, let $\hat T$ be the Tomita map, which is the closure of
$\hat\Lambda(\mf n_{\hat\varphi}\cap\mf n_{\hat\varphi}^*) \rightarrow L^2(\G),
\hat\Lambda(x)\mapsto \hat\Lambda(x^*)$.  For further details, see
\cite[Chapter~1]{tak2} (where this map is denoted by $S$, a notation which
we avoid, as it clashes with the antipode).

\begin{proposition}\label{prop:eight}
Let $X = \{ \omega\in L^1(\G)\cap L^2(\G) : \hat\Lambda(\lambda(\omega)) \in
D(\hat T^*) \}$.  Then $X$ is dense in $L^1(\G)$, and $\hat\Lambda(\lambda(X))$
is dense in $L^2(\G)$.  For $x\in D(S)^*$ and $\omega\in X$, we have that
$x\omega\in X$, and $\hat T^* \hat\Lambda(\lambda(x\omega))
= S(x^*) \hat T^* \hat\Lambda(\lambda(\omega))$.
\end{proposition}
\begin{proof}
Let $\omega\in L^1(\G)\cap L^2(\G)$ be given by Lemma~\ref{lem:one}.
As $\lambda(\rho_t(\omega)) = \hat\sigma_t(\lambda(\omega))$, an analytic
continuation argument shows that $\lambda(\omega) \in D(\hat\sigma_i)$
with $\hat\sigma_i(\lambda(\omega)) = \lambda(\rho_i(\omega))$.  As also
$\rho_i(\omega)\in L^1_\sharp(\G)$ and $\rho_i(\omega)^\sharp
\in L^1(\G)\cap L^2(\G)$, we have that
\begin{align*} \hat\Lambda\big( \lambda(\rho_i(\omega)^\sharp) \big)
&= \hat\Lambda\big( \lambda(\rho_i(\omega))^* \big)
= \hat T \hat\Lambda\big( \lambda(\rho_i(\omega)) \big)
= \hat T \hat\Lambda\big( \sigma_i(\lambda(\omega)) \big) \\
&= \hat T \hat\nabla^{-1} \hat\Lambda\big( \lambda(\omega) \big)
= \hat T^* \hat\Lambda\big( \lambda(\omega) \big). \end{align*}
As such $\omega$ are dense in $L^1(\G)$, it follows that $X$ is dense.
As the collection of vectors $\hat\Lambda(\lambda(\omega))$ is dense in
$L^2(\G)$, it also follows that $\hat\Lambda(\lambda(X))$ is dense.

Now, $D(S)^* = D(\tau_{i/2})$.  For $x\in L^\infty(\G)$ and $t\in\mathbb R$,
$\tau_t$ is implemented as $\tau_t(x) = \hat\nabla^{it} x \hat\nabla^{-it}$.
Thus, that $x\in D(\tau_{i/2})$ means that $\hat\nabla^{-1/2} x \hat\nabla^{1/2}$
extends to a bounded operator, namely $\tau_{i/2}(x)$.  Thus also
$\hat J \hat\nabla^{-1/2} x \hat\nabla^{1/2}\hat J \subseteq S(x^*)$.
As $\hat T^* = \hat J \hat\nabla^{-1/2} = \hat\nabla^{1/2} \hat J$, we conclude that
$\hat T^* x \hat T^* \subseteq S(x^*)$.  Thus, for $x\in S(x)^*$ and $\omega\in X$,
\[ S(x^*) \hat T^* \hat\Lambda(\lambda(\omega)) = \hat T^* x \hat T^*\hat T^*
\hat\Lambda(\lambda(\omega))
= \hat T^* x\hat\Lambda(\lambda(\omega))
= \hat T^* \hat\Lambda(\lambda(x\omega)), \]
as required.  Here we used that $\hat T^* = (\hat T^*)^{-1}$, see
\cite[Lemma~1.5, Chapter~1]{tak2}.  
\end{proof}

Let us make one final definition:
\[ L^1_\sharp(\G) \cap L^2_\sharp(\G) =
\{  \omega\in L^1(\G)\cap L^2(\G) : \omega\in L^1_\sharp(\G)
\text{ and } \omega^\sharp \in L^1(\G)\cap L^2(\G) \}. \]
Then \cite[Proposition~2.6]{kvvn} shows that $L^1_\sharp(\G) \cap L^2_\sharp(\G)$
is dense in $L^1(\G)$, and $\hat\Lambda(\lambda(L^1_\sharp(\G) \cap
L^2_\sharp(\G)))$ is dense in $L^2(\G)$.

\begin{proposition}\label{prop:seven}
Let $\omega_1\in L^1_\sharp(\G)\cap L^2_\sharp(\G)$, and
let $\omega_2\in X$ be as in the previous proposition.  Let $x\in L^\infty(\G)$,
set $\xi = \hat\Lambda(\lambda(\omega_1))$ and
$\eta = \hat T^*\hat\Lambda(\lambda(\omega_2))$, and set
$\hat\omega = \hat\omega_{x\xi,\eta}$.  Then
$\hat\omega \in L^1(\hat\G)\cap L^2(\G)$
and $\Lambda(\hat\lambda(\hat\omega)) = \hat\Lambda(\lambda((x\omega_1)\omega_2))$.
When we take $x=1$, as $\omega_1,\omega_2$ vary, such $\hat\omega$ are dense
in $L^1(\hat\G)$.
\end{proposition}
\begin{proof}
Let $\hat x\in \mf n_{\hat\varphi}$.  As $\lambda(x\omega_1) \in
\mf n_{\hat\varphi}$ and $\mf n_{\hat\varphi}$ is a left ideal, both
$\lambda(x\omega_1)^* \hat x \in \mf n_{\hat\varphi}$
and $\hat x^* \lambda(x\omega_1) \in \mf n_{\hat\varphi}$, so that
$\hat x^* \lambda(x\omega_1) \in \mf n_{\hat\varphi} \cap \mf n_{\hat\varphi}^*$.
Thus $\hat\Lambda(\hat x^* \lambda(x\omega_1)) \in D(\hat T)$ and
$\hat T \hat\Lambda(\hat x^* \lambda(x\omega_1))
= \hat\Lambda(\lambda(x\omega_1)^* \hat x)$.  Thus
\begin{align*}
\ip{\hat x^*}{\hat\omega} &= ( \hat x^* x \xi|\eta)
= \big( \hat x^* x \hat\Lambda(\lambda(\omega_1)) \big|
   \hat T^*\hat\Lambda(\lambda(\omega_2)) \big) 
= \big( \hat x^* \hat\Lambda(\lambda(x\omega_1)) \big|
   \hat T^*\hat\Lambda(\lambda(\omega_2)) \big) \\
&= \big( \hat\Lambda(\lambda(\omega_2)) \big|
   \hat T\hat\Lambda(\hat x^*\lambda(x\omega_1)) \big)
= \big( \hat\Lambda(\lambda(\omega_2)) \big|
   \hat\Lambda(\lambda(x\omega_1)^*\hat x) \big) \\
&= \hat\varphi\big( \hat x^* \lambda(x\omega_1) \lambda(\omega_2) \big)
= \big( \hat\Lambda(\lambda((x\omega_1)\omega_2)) \big| \hat\Lambda(\hat x) \big).
\end{align*}
As $\hat x\in\mf n_{\hat\varphi}$ was arbitrary, this shows that
$\hat\omega\in L^1(\hat\G)\cap L^2(\G)$ with $\Lambda(\hat\lambda(\hat\omega))
= \hat\Lambda(\lambda((x\omega_1)\omega_2))$ as claimed.
As above, the collection of allowed $\xi$ is dense in $L^2(\G)$, and
by Proposition~\ref{prop:eight}, the collection of allowed $\eta$ is also
dense in $L^2(\G)$.  Hence certainly the collection of thus constructed
$\hat\omega$ will be dense in $L^1(\hat \G)$.
\end{proof}

\subsection{Coefficients of representations}

Let $\pi:L^1(\G)\rightarrow\mc B(H)$ be a completely bounded (anti-)homomorphism
with associated co(-anti-)representation $V_\pi$.  Given $\alpha,\beta\in H$,
a \emph{co-efficient} of $\pi$ is the operator
\begin{equation} \label{eqn:coefficient}
T^\pi_{\alpha,\beta} = (\iota \otimes \omega_{\alpha,\beta})V_\pi \in L^\infty(\G).
\end{equation}
Equivalently, $T^\pi_{\alpha,\beta}$ is determined by the dual pairing \[\ip{T^\pi_{\alpha,\beta}}{\omega} = \big(\pi(\omega)\alpha \big|\beta\big) \qquad (\omega \in L^1(\G)).\]  Note that the latter definition of $T^\pi_{\alpha, \beta}$ makes sense even if $\pi$ is bounded, but not necessarily completely bounded -- a situation we will address in Section \ref{sec:importance_cb}.
 
In this subsection and the next one, we study the analytic structure of co-efficients of completely bounded representations $\pi:L^1(\G)\rightarrow\mc B(H)$.  We show that if $\pi^*$ is also completely bounded, then co-efficients of $\pi$ give rise to completely bounded multipliers of $L^1(\hat\G)$, a tool which will allow us to prove the main theorems of this section (Theorems \ref{thm:main2} and \ref{thm:summ1})  below. 
%In this section, we show that if $\pi^*$ is also completely bounded,
%then co-efficients of $\pi$ give rise to completely bounded multipliers
%of $L^1(\hat\G)$, a tool which will allow us to prove our main theorem below.

\begin{proposition}\label{prop:nine}
Suppose that both $\pi$ and $\pi^*$ extend to completely bounded maps
$L^1(\G)\rightarrow\mc B(H)$.  For every $\alpha,\beta \in H$, we have that
$T^{\pi^*}_{\alpha,\beta} \in D(S)$ with
$S(T^{\pi^*}_{\alpha,\beta})^* = T^\pi_{\beta,\alpha}$.
\end{proposition}
\begin{proof}
Let $\omega \in L^1_\sharp(\G)$, so that
\[ \ip{T^\pi_{\beta,\alpha}}{\omega^\sharp}
= \big( \pi(\omega^\sharp)\beta \big| \alpha \big)
= \overline{ \big( \pi(\omega^\sharp)^*\alpha \big| \beta \big) }
= \overline{ (\pi^*(\omega)\alpha|\beta) }
= \overline{ \ip{T^{\pi^*}_{\alpha,\beta}}{\omega} }. \]
Comparing this with the definition of $\sharp$, we might hope that
$T^\pi_{\beta,\alpha} \in D(S)$ with $S(T^\pi_{\beta,\alpha})^*
= T^{\pi^*}_{\alpha,\beta}$.  This is indeed true, but a little care is
needed, see Proposition~\ref{prop:aptwo} in the appendix.
As $S\circ *\circ S = *$,
we also see that $T^{\pi^*}_{\alpha,\beta} \in D(S)$ with
$S(T^{\pi^*}_{\alpha,\beta})^* = T^\pi_{\beta,\alpha}$, as claimed.
\end{proof}

Given a co-efficient $T^\pi_{\alpha,\beta}$, arguing as in Section~\ref{sec:hatp}
and as in the proof of Theorem~\ref{thm:inv_case}, we see that
\begin{align*}
\Delta(T^\pi_{\alpha,\beta}) = (\iota\otimes\iota\otimes\omega_{\alpha,\beta})
(V_{\pi, 13} V_{\pi,23})
= \sum_i (\iota\otimes\omega_{f_i,\beta})(V_\pi) \otimes
(\iota\otimes\omega_{\alpha,f_i})(V_\pi)
= \sum_i T^\pi_{f_i,\beta} \otimes T^\pi_{\alpha,f_i},
\end{align*}
the sum converging $\sigma$-weakly.  This observation, combined with the
previous proposition, shows that the hypotheses of the next theorem are not
so outlandish; compare with the proof of Theorem~\ref{thm:main2} below.

\begin{theorem}\label{thm:one}
Let $\G$ be a locally compact quantum group, let $x\in L^\infty(\G)$, and suppose
that $\sigma\Delta(x) = \sum_i b_i \otimes a_i \in L^\infty(\G)
\otimes_{eh} L^\infty(\G)$.
Suppose furthermore that each $b_i\in D(S)^*$, and $\sum_i S(b_i^*)^* S(b_i^*)
< \infty$.  Then, for $\xi,\eta\in L^2(\G)$,
\[ x \hat\lambda(\hat\omega_{\xi,\eta})
= \hat\lambda \Big( \sum_i \hat\omega_{a_i\xi,S(b_i^*)\eta} \Big), \]
the sum converging absolutely in $L^1(\hat\G)$.
\end{theorem}
\begin{proof}
Let $\omega_1,\omega_2\in L^1(\G)$ be given by Proposition~\ref{prop:seven}.
Let $\xi = \hat\Lambda(\lambda(\omega_1)), \eta = \hat T^* \hat\Lambda(
\lambda(\omega_2))$; recall that such choices are both dense in $L^2(\G)$.
Then $\hat\omega_{\xi,\eta} \in L^1(\hat\G) \cap L^2(\G)$ and
$\Lambda(\hat\lambda(\hat\omega_{\xi,\eta})) = \hat\Lambda(\lambda(\omega_1
\omega_2))$.

For each $i$, let $\xi_i = a_i(\xi)$ and $\eta_i = S(b_i^*)(\eta)$.
Then
\[ \sum_i \|\xi_i\|^2 = \sum_i \ip{a_i^*a_i}{\omega_{\xi,\xi}}
\leq \Big\| \sum_i a_i^* a_i \Big\| \|\xi\|^2. \]
Similarly,
\[ \sum_i \|\eta_i\|^2
\leq \Big\| \sum_i S(b_i^*)^* S(b_i^*) \Big\| \|\eta\|^2. \]
It follows that $\sum_i \hat\omega_{\xi_i,\eta_i}$ converges absolutely
in $L^1(\hat\G)$.

By Proposition~\ref{prop:eight}, $\eta_i = \hat T^*\hat\Lambda(\lambda(
b_i \omega_2))$, and also $\xi_i = a_i\hat\Lambda(\lambda(\omega_1))$.
By Proposition~\ref{prop:seven},
\[ \Lambda(\hat\lambda(\hat\omega_{\xi_i,\eta_i}))
= \hat\Lambda\big(\lambda\big( (a_i\omega_1)(b_i\omega_2) \big)\big). \]
Thus, for $y\in \mf n_{\varphi}$,
\begin{align*}
\sum_i \big( \Lambda(\hat\lambda(\hat\omega_{\xi_i,\eta_i})) \big| \Lambda(y) \big)
&= \sum_i \big( \hat\Lambda\big(\lambda\big(
  (a_i\omega_1)(b_i\omega_2) \big)\big)  \big| \Lambda(y) \big) \\
&= \sum_i \ip{y^*}{(a_i\omega_1)(b_i\omega_2)}
= \sum_i \ip{\Delta(y^*)(a_i\otimes b_i)}{\omega_1\otimes\omega_2} \\
&= \ip{\Delta(y^*x)}{\omega_1\otimes\omega_2}
= \ip{y^*}{x(\omega_1\omega_2)}
= \big( \hat\Lambda(\lambda(x(\omega_1\omega_2))) \big| \Lambda(y) \big).
\end{align*}
Thus,
\begin{align*} \Lambda\big( x \hat\lambda(\hat\omega_{\xi,\eta}) \big)
&= x \hat\Lambda(\lambda(\omega_1\omega_2))
= \hat\Lambda(\lambda(x(\omega_1\omega_2)))
= \sum_i \Lambda(\hat\lambda(\omega_{\xi_i,\eta_i})), \end{align*}
which completes the proof, as $\Lambda$ injects, and such $\xi,\eta$
are dense.
\end{proof}

We remark that we could weaken the hypothesis that $\sigma\Delta(x)\in
L^\infty(\G) \otimes_{eh} L^\infty(\G)$ to just requiring that
$\Delta(x) = \sum_i a_i \otimes b_i$, the sum converging $\sigma$-weakly in
$L^\infty(\G)\vnten L^\infty(\G)$, and with $\sum_i a_i^*a_i < \infty$
and $\sum_i S(b_i^*)^* S(b_i^*) < \infty$.

We can interpret this result in terms of \emph{left multipliers} or
\emph{centralisers} on $L^1(\hat \G)$.  We shall follow \cite{dm2} (see also
\cite{jnr}, but be aware of differing language).  A left multiplier of
$L^1(\hat\G)$ is a linear map $L:L^1(\hat\G) \rightarrow L^1(\hat\G)$ with
$L(\hat\omega_1\hat\omega_2) = L(\hat\omega_1) \hat\omega_2$ for
each $\hat\omega_1, \hat\omega_2\in L^1(\hat\G)$.  If additionally $L$ is
completely bounded, then an equivalent condition is that the adjoint
$L^*:L^\infty(\hat\G) \rightarrow L^\infty(\hat\G)$ satisfies
$\hat \Delta L^* = (L^*\otimes\iota)\hat \Delta$.

Given data as in the above theorem, we may define a $\sigma$-weakly continuous linear map
$L^*:L^\infty(\hat\G) \rightarrow \mc B(L^2(\G))$ by
\begin{eqnarray} \label{eqn:mult_adjoint} L^*(\hat x) = \sum_i S(b_i^*)^* \hat x a_i
\qquad (\hat x \in L^\infty(\hat\G)).
\end{eqnarray}
Then clearly the preadjoint $L:\mc B(L^2(\G))_* \rightarrow L^1(\hat\G)$
satisfies $\hat\lambda(L\omega_{\xi,\eta}) = x \hat\lambda(\hat\omega_{\xi,\eta}) 
$ for all $\xi,\eta\in L^2(\G)$.
Thus $L$ factors through the quotient map $\mc B(L^2(\G))_* \rightarrow
L^1(\hat\G)$ and induces a map on $L^1(\hat\G)$, still denoted by $L$.  Hence we can regard
$L^*$ as a map on $L^\infty(\hat\G)$.  We record the following
for use later.

\begin{proposition}\label{prop:four}
With notation as above, $L^*\in\mc{CB}(L^\infty(\hat\G))$,
\[ \|L^*\|_{cb} \le \Big\| \sum_{i} S(b_i^*)^*S(b_i^*)\Big\|^{1/2}
\Big\|\sum_i a_i^*a_i\Big\|^{1/2}, \]
and $(L^*\otimes\iota)(\hat W) = (1\otimes x)\hat W$.  
\end{proposition}
\begin{proof}
The norm estimate is a standard calculation from the definition of $L^*$
given by (\ref{eqn:mult_adjoint}).  The last assertion follows from
\cite[Proposition 2.4]{dm2}; indeed, this is just a
calculation using the definition of $\hat\lambda$ and that
$x \hat\lambda(\hat\omega_{\xi,\eta}) = \hat\lambda(L\hat\omega_{\xi,\eta})$.
\end{proof}

We remark that it follows from \cite{dm2} that $x\in C^b(\G)$
(compare with Theorem~\ref{thm:summ1} below) and that $x\in D(S)$.

\subsection{When we get an invertible corepresentation}

We will work in a little generality, and deal with possibly degenerate
homomorphisms.  Given a homomorphism $\pi:A\rightarrow\mc B(H)$ of a Banach
algebra $A$ on a Hilbert space $H$, the \emph{essential space of $\pi$} is $H_e$,
the closed linear span of $\{ \pi(a)\xi :a\in A,\xi\in H\}$.

The following shows that if $\pi$ and $\pi^*$ are completely bounded, then
$V_\pi$ is an invertible operator (suitably interpreted in the degenerate
case), and thus we have a converse to Theorem~\ref{thm:inv_case}.

\begin{theorem}\label{thm:main2}
Let $\pi:L^1(\G)\rightarrow\mc B(H)$ be a completely bounded
representation such that $\pi^*$ extends to a completely bounded representation.
Letting $H_e$ be the essential space of $\pi$, we have that $V_{\check\pi} V_\pi
= V_\pi V_{\check\pi}$ is a (not necessarily orthogonal) projection of
$L^2(\G)\otimes H$ onto $L^2(\G) \otimes H_e$.  Furthermore, both $V_\pi$ and
$V_{\check\pi}$ have ranges equal to $L^2(\G) \otimes H_e$.  Hence $V_\pi$ and
$V_{\check\pi}$ restrict to operators on $L^2(\G) \otimes H_e$, and are mutual
inverses in $\mc B(L^2(\G) \otimes H_e)$.
\end{theorem}
\begin{proof}
Fix $\alpha,\beta\in H$.
By Proposition~\ref{prop:two}, we know that $V_{\tilde\pi} = V_\pi^*$,
and thus $T^{\tilde\pi}_{\alpha,\beta} = (T^\pi_{\beta,\alpha})^*$.
As in the discussion before Theorem~\ref{thm:one},
\[ \sigma \Delta(T^{\tilde\pi}_{\alpha,\beta})
= \sigma \Delta(T^{\pi}_{\beta, \alpha})^*  = \sum_i (T^\pi_{\beta,f_i})^*
\otimes (T^\pi_{f_i,\alpha})^* 
= \sum_i T^{\tilde\pi}_{f_i,\beta} \otimes  T^{\tilde\pi}_{\alpha,f_i}. \]
As $\pi^*$ is completely bounded, by Proposition~\ref{prop:nine},
\[ (T^{\tilde\pi}_{f_i,\beta})^* = T^\pi_{\beta,f_i} \in D(S)
\quad\text{with}\quad S\big( (T^{\tilde\pi}_{f_i,\beta})^* \big)
= S\big( T^\pi_{\beta,f_i}\big)
= (T^{\pi^*}_{f_i,\beta})^* = T^{\check\pi}_{\beta,f_i}, \]
where the last equality uses Proposition~\ref{prop:two}.
Moreover, \begin{align*}
\sum_i S\big( (T^{\tilde\pi}_{f_i,\beta})^* \big)^*S\big( (T^{\tilde\pi}_{f_i,\beta})^* \big) &=\sum_i T^{\pi^*}_{f_i,\beta}(T^{\pi^*}_{f_i,\beta})^*
= \sum_i (\iota\otimes\omega_{f_i,\beta})(V_{\pi^*})\big(
(\iota\otimes\omega_{\beta,f_i})(V_{\pi^*})\big)^* \\
&= \sum_i (\iota\otimes\omega_{f_i,\beta})(V_{\pi^*})
(\iota\otimes\omega_{f_i,\beta})(V_{\pi^*}^*)
= (\iota\otimes\omega_\beta)(V_{\pi^*}V_{\pi^*}^*), 
\end{align*}
and similarly $\sum_i(T^{\tilde\pi}_{\alpha,f_i})^*T^{\tilde\pi}_{\alpha,f_i} =  (\iota\otimes\omega_\alpha)(V_{\pi}V_{\pi}^*)$.
We can hence apply Theorem~\ref{thm:one} and Proposition~\ref{prop:four}
to conclude that if we define
\begin{eqnarray} \label{eqn:mult_adjoint2}
L^*:L^\infty(\hat\G) \rightarrow \mc B(L^2(\G));
\quad \hat x \mapsto \sum_i T^{\pi^*}_{f_i,\beta} \hat x
T^{\tilde\pi}_{\alpha,f_i}, 
\end{eqnarray}
then $L^* \in \mc {CB}(L^\infty(\hat\G))$,
and $(L^*\otimes\iota)(\hat W) = (1\otimes T^{\tilde\pi}_{\alpha,\beta})\hat W$.
As $\hat W= \sigma W^* \sigma$, this shows that $(\iota\otimes L^*)(W^*) W
= T^{\tilde\pi}_{\alpha,\beta}\otimes 1$.  That is,
\[ T^{\tilde\pi}_{\alpha,\beta}\otimes 1
= \sum_i (1\otimes T^{\pi^*}_{f_i,\beta}) W^*
   (1\otimes T^{\tilde\pi}_{\alpha,f_i}) W
= \sum_i (1\otimes T^{\pi^*}_{f_i,\beta}) \Delta(T^{\tilde\pi}_{\alpha,f_i}). \]

So, let $\omega_1,\omega_2\in L^1(\G)$, and consider
\begin{align*} \sum_i & \ip{(1\otimes T^{\pi^*}_{f_i,\beta})
   \Delta(T^{\tilde\pi}_{\alpha,f_i})}{\omega_1\otimes\omega_2}
= \sum_i \ip{T^{\tilde\pi}_{\alpha,f_i}}{\omega_1
   \big( \omega_2 T^{\pi^*}_{f_i,\beta} \big)}
= \sum_i \big( \tilde\pi\big( \omega_1
   \big( \omega_2 T^{\pi^*}_{f_i,\beta} \big) \big) \alpha \big| f_i \big) \\
&= \sum_i \big( \tilde\pi\big( \omega_2 T^{\pi^*}_{f_i,\beta} \big) 
   \tilde\pi(\omega_1)\alpha \big| f_i \big)
= \sum_i \ip{ T^{\pi^*}_{f_i,\beta}
   T^{\tilde\pi}_{\tilde\pi(\omega_1)\alpha,f_i} }{\omega_2} \\
&= \ip{ (\iota\otimes\omega_{\tilde\pi(\omega_1)\alpha,\beta})
   (V_{\pi^*} V_{\tilde\pi})}{\omega_2}.
\end{align*}
However, we know that this is equal to $\ip{T^{\tilde\pi}_{\alpha,\beta}\otimes 1}
{\omega_1\otimes\omega_2}$.  So we conclude that
\[ (\iota\otimes\omega_{\tilde\pi(\omega)\alpha,\beta})
   (V_{\pi^*} V_{\tilde\pi}) = \ip{T^{\tilde\pi}_{\alpha,\beta}}{\omega} 1
   = \big( \tilde\pi(\omega)\alpha \big| \beta \big) 1
\qquad (\omega\in L^1(\G), \alpha,\beta\in H). \]
Letting $\tilde H_e$ be the essential space of $\tilde\pi$, we conclude that
$V_{\pi^*} V_{\tilde\pi} = 1$ on $L^2(\G) \otimes \tilde H_e \subseteq L^2(\G)
\otimes H$.

As $L^1_\sharp(\G)$ is dense in $L^1(\G)$, we see that the essential space
for $\pi^*$ also equals $\tilde H_e$.  Letting $(e_i)$ be an orthonormal basis
for $L^2(\G)$, a simple calculation shows that
\[ V_{\pi^*}(\xi\otimes\gamma) = \sum_i e_i \otimes \pi^*(\omega_{\xi,e_i})(\gamma)
 \in L^2(\G) \otimes \tilde H_e \qquad (\xi\in L^2(\G), \gamma\in H). \]
So $V_{\pi^*}$ has range contained in $L^2(\G) \otimes \tilde H_e$.  It follows
that $V_{\pi^*} V_{\tilde\pi}$ is actually a (not necessarily orthogonal)
projection from $L^2(\G)\otimes H$ onto $L^2(\G)\otimes \tilde H_e$.
Furthermore, the range of $V_{\pi^*}$ must actually be $L^2(\G)\otimes \tilde H_e$.

Now set $\phi = \pi^*$, so $\phi$ is also a completely bounded homomorphism
$L^1(\G)\rightarrow \mc B(H)$.  Thus the same argument now applied to $\phi$
shows that $V_{\phi^*} V_{\tilde\phi} = V_\pi V_{\pi^*}^* = V_\pi V_{\check\pi}$
is a (not necessarily orthogonal) projection from $L^2(\G)\otimes H$ onto
$L^2(\G)\otimes H_e$ where $H_e$ is the essential space of $\tilde\phi$,
which agrees with the essential space for $\phi^* = \pi$.  We also conclude
that the range of $V_{\phi^*} = V_\pi$ is $L^2(\G)\otimes H_e$.

Following \cite[Section~4]{kvvn}, the \emph{opposite quantum group}
to $\G$ is $\G^\op$, where $L^\infty(\G^\op) = L^\infty(\G)$ and
$\Delta^\op = \sigma\Delta$.  That is, we reverse the multiplication in
$L^1(\G)$ to get $L^1(\G^\op)$.  Then $R^\op = R$ and $\tau^\op_t = \tau_{-t}$
for each $t$.  Thus $S^\op = R^\op \tau^\op_{-i/2} = R \tau_{i/2} = S^{-1}
= * \circ S \circ *$.  For $x\in D(S^\op) = D(S)^*$, and $\omega\in
L^1_\sharp(\G^\op)$,
\[ \ip{x}{\omega^{\sharp,\op}} = \overline{ \ip{S^\op(x)^*}{\omega} }
= \overline{ \ip{S(x^*)}{\omega} } = \ip{S(x^*)^*}{\omega^*}. \]
It follows that $\omega^*\in L^1_\sharp(\G)$, as for each $y\in D(S)$,
setting $x=y^*\in D(S)^*$, we see that
\[ \overline{\ip{S(y)^*}{\omega^*}} = \ip{S(x^*)}{\omega}
= \overline{ \ip{x}{\omega^{\sharp,\op}} }
= \ip{y}{(\omega^{\sharp,\op})^*}. \]
So $(\omega^*)^\sharp = (\omega^{\sharp,\op})^*$, and reversing this
argument shows that $L^1_\sharp(\G) = L^1_\sharp(\G^\op)^*$.

Thus, if we set $\phi = \tilde\pi: L^1(\G^\op) \rightarrow \mc B(H)$ then
$\phi$ is a completely bounded representation, and $\tilde\phi = \pi$, 
so that the essential space of $\phi$ is $H_e$.
For $\omega\in L^1_\sharp(\G^\op)$, we have that $\omega^*\in L^1_\sharp(\G)$,
and $\phi^*(\omega) = \phi(\omega^{\sharp,\op})^*
= \tilde\pi( ((\omega^*)^\sharp)^* )^*
= \pi( (\omega^*)^\sharp ) = \check\pi(\omega)$.  Hence, as $\check\pi$
is completely bounded, also $\phi^*$ is completely bounded; and so
$\phi^* = \check\pi$ and similarly $\check\phi = \pi^*$.
Applying the previous argument to $\phi$, we conclude that
$V_{\phi^*} V_{\tilde\phi} = V_{\check\pi} V_\pi$ is a (not necessarily orthogonal)
projection onto $L^2(\G) \otimes H_e$.  We also see that the range of
$V_{\phi^*} = V_{\check\pi}$ is $L^2(\G) \otimes H_e$.

Finally, we also have that $V_\phi V_{\check\phi} = V_{\tilde\pi} V_{\pi^*}$
is a (not necessarily orthogonal) projection onto $L^2(\G) \otimes \tilde H_e$.
It follows that
\[ \ker\big( V_\pi V_{\check\pi} \big)
= \Image(V_{\check\pi}^* V_\pi^*)^\perp
= \Image(V_{\pi^*} V_{\tilde\pi})^\perp
= \big( L^2(\G) \otimes \tilde H_e \big)^\perp, \]
and analogously, also $\ker( V_{\check\pi}V_\pi )
= \Image( V_{\tilde\pi}V_{\pi^*})^\perp
= ( L^2(\G) \otimes \tilde H_e)^\perp$.  So we conclude that
$\ker( V_\pi V_{\check\pi} ) = \ker( V_{\check\pi}V_\pi )$.  So
$V_\pi V_{\check\pi}$ and $V_{\check\pi} V_\pi$ are projections with
the same range and kernel, and hence are equal.
\end{proof}

The above theorem says that if both $\pi$ and $\pi^*$ are
completely bounded and non-degenerate, then $V_\pi$ is an invertible
corepresentation of $L^\infty(\G)$, and, informally,
\[ (S\otimes\iota)(V_\pi) = V_\pi^{-1}. \]
This is well-known in, for example, the theory of algebraic compact
quantum groups (compare with \cite[Proposition~3.1.7(iii)]{tt}).
It is interesting that our arguments
seem to require a lot of structure-- for example, by using the duality
theory for locally compact quantum groups.

From Proposition~\ref{prop:four} and the proof of Theorem~\ref{thm:main2},
we obtain the following corollary, showing that the co-efficients
$T^{\tilde\pi}_{\alpha, \beta}$ considered in Theorem~\ref{thm:main2}
naturally induce completely bounded multipliers on $L^1(\hat \G)$.

\begin{corollary} \label{cor:cbmults}
Let $\pi:L^1(\G)\rightarrow\mc B(H)$ be a completely bounded
representation such that $\pi^*$ extends to a completely bounded representation.
Then for any $\alpha, \beta \in H$, the co-efficient $T^{\tilde \pi}_{\alpha,\beta}$
represents a completely bounded left multiplier $L:L^1(\hat \G) \to L^1(\hat\G)$,
determined by 
\[ \hat \lambda(L\hat \omega) = T^{\tilde\pi}_{\alpha,\beta}\hat\lambda(\hat \omega)
\qquad (\hat \omega \in L^1(\hat \G)). \]
Moreover, $\|L\|_{cb} \le \|\pi\|_{cb}\|\pi^*\|_{cb} \|\alpha\|\|\beta\|$.
\end{corollary}  
\begin{proof}
The fact that $L$ is a completely bounded left multiplier was already observed in
the proof of Theorem~\ref{thm:main2}.  Moreover, from Proposition~\ref{prop:four} and the proof of Theorem~\ref{thm:main2}, we deduce that 
\begin{align*} \|L\|_{cb} &= \|L^*\|_{cb} \leq
\|(\iota \otimes \omega_{\beta})(V_{\pi^*}V_{\pi^*}^*)\|^{1/2}
\|(\iota \otimes \omega_{\alpha})(V_{\tilde\pi}^*V_{\tilde\pi})\|^{1/2}
\leq \|\pi\|_{cb}\|\pi^*\|_{cb} \|\alpha\|\|\beta\|. \end{align*}
\end{proof} 

When $\G$ is a locally compact group, the above result is classical and is due to
de Canniere and Haagerup \cite[Theorem 2.2]{dh}.
(See also \cite[Theorem 3]{bs} for when $\G$ is the dual of a locally compact group.)

Let $G$ be a locally compact group, $\pi:L^1(G) \to \mc B(H)$ a bounded representation,
and $\pi_0:G \to \mc B(H)$ the uniformly bounded representation associated to $\pi$.
The correspondence $\pi \leftrightarrow \pi_0$ given in Section~\ref{sec:intro}
implies that each co-efficient operator $T^\pi_{\alpha,\beta} \in L^\infty(G)$ is
actually (a.e.) equal to the continuous function $\varphi^\pi_{\alpha,\beta} \in C^b(G)$,
$s \mapsto \big(\pi_0(s)\alpha|\beta\big)$.

Let $\mc B_0(H)$ be the compact operators on $H$, so that $M(\mc B_0(H)) = \mc B(H)$.
We hence get the \emph{strict topology} on $\mc B(H)$, where a bounded net
$(T_\alpha)$ is strictly-null if and only if the nets $(T_\alpha x)$ and
$(xT_\alpha)$ are norm-null, for each $x\in\mc B_0(H)$.
One consequence of the above correspondence is that the corepresentation
$V_\pi \in L^\infty(G)\vnten\mc B(H)$ is actually a member of
$C^b_{str}(G,B(H))$, the space of bounded continuous maps
$G\rightarrow\mc B(H)$, where $B(H)$ is given the strict topology.
Indeed, $V_\pi$ corresponds to the map $G\rightarrow\mc B(H); s\mapsto \pi_0(s)$.

When $\G$ is a locally compact quantum group, the natural replacement
for $C^b_{str}(G,\mc B(H))$ is the multiplier algebra
$M(C_0(\G) \otimes \mc B_0(H))$.  This follows, as $M(C_0(G)\otimes
\mc B_0(H))$ is isomorphic to $C^b_{str}(G,\mc B(H))$.  Indeed, given
$F\in C^b_{str}(G,\mc B(H))$ and $f\in C_0(G,\mc B_0(H)) = C_0(G)
\otimes \mc B_0(H))$, clearly the pointwise products $Ff, fF$ are
members of $C_0(G,\mc B_0(H))$, and so $F$ is a multiplier, and
a partition of unity argument shows that every multiplier arises in this way.

The following gives a summary of the results of this section.  The argument
about multiplier algebras should be compared with \cite[Section~4]{woro},
which in turn is inspired by \cite[Page~441]{bs2}.

\begin{theorem}\label{thm:summ1}
Let $\G$ be a locally compact quantum group, and let $\pi:L^1(\G)
\rightarrow \mc B(H)$ be a non-degenerate, completely bounded representation.
Then the following are equivalent:
\begin{enumerate}
\item $\check\pi$ extends to a completely bounded map $L^1(\G)\rightarrow
\mc B(H)$;
\item $V_\pi$ is invertible.
\end{enumerate}
In this case, $V_\pi^{-1} = V_{\check\pi}$, and
$V_\pi \in M(C_0(\G) \otimes \mc B_0(H))$.
\end{theorem}
\begin{proof}
The equivalences have been established by Theorem~\ref{thm:inv_case}
and Theorem~\ref{thm:main2}.  Suppose now that $V_\pi^{-1}$ exists.
Recall that the coproduct on $L^\infty(\G)$ is implemented by the
fundamental unitary by $\Delta(x) = W^*(1\otimes x)W$ for $x\in L^\infty(\G)$.
Furthermore, $W\in M(C_0(\G)\otimes\mc B_0(L^2(\G)))$.  Then
\[ V_{\pi,13} = (\Delta\otimes\iota)(V_\pi) V_{\pi,23}^{-1}
= W^*_{12} V_{\pi,23} W_{12} V_{\pi,23}^{-1}
\in M\big( C_0(\G) \otimes B_0(L^2(\G)) \otimes \mc B_0(H) \big). \]
Thus $V_\pi \in M(C_0(\G)\otimes \mc B_0(H))$ as claimed.
\end{proof}

\subsection{Application to (co)isometric corepresentations}
\label{sec:coisocase}

A corepresentation $U\in L^\infty(\G) \vnten \mc B(H)$ is a
\emph{coisometric corepresentation} if $U$ is a coisometry, namely $UU^*=1$.
As $\Delta$ is a $*$-homomorphism, we see that $U^*$ is a co-anti-representation.

\begin{proposition}
Let $U\in L^\infty(\G) \vnten \mc B(H)$ be a coisometric corepresentation,
associated to $\pi:L^1(\G)\rightarrow\mc B(H)$.  Then $\check\pi$ is
a completely bounded representation, with $V_{\check\pi} = U^*$.
\end{proposition}
\begin{proof}
This follows almost exactly as the proof of Theorem~\ref{thm:inv_case}.
Let $\alpha,\beta\in H$, let $(f_i)$ be an orthonormal basis of $H$, and set
\[ x_i = (\iota\otimes\omega_{f_i,\beta})(U), \quad
y_i = (\iota\otimes\omega_{\alpha,f_i})(U^*). \]
Then as before, we see that
\[ \sum_i \Delta(x_i)(1\otimes y_i)
= (\iota\otimes\iota\otimes\omega_{\alpha,\beta})(U_{13} U_{23} U_{23}^*)
= (\iota\otimes\omega_{\alpha,\beta})(U) \otimes 1, \]
and similarly
\[ \sum_i (1\otimes x_i)\Delta(y_i)
= (\iota\otimes\iota\otimes\omega_{\alpha,\beta})(U_{23} U_{23}^* U_{13}^*)
= (\iota\otimes\omega_{\alpha,\beta})(U^*)\otimes 1. \]
Hence $(\iota\otimes\omega_{\alpha,\beta})(U) \in D(S)$ with
$S((\iota\otimes\omega_{\alpha,\beta})(U)) =
(\iota\otimes\omega_{\alpha,\beta})(U^*)$.  As before, it now follows
that $\check\pi$ is completely bounded with $V_{\check\pi} = U^*$.
\end{proof}

The following is now immediate from Theorem~\ref{thm:main2}.

\begin{corollary}
Let $U\in L^\infty(\G) \vnten \mc B(H)$ be a coisometric corepresentation.
Then $U$ is unitary.
\end{corollary}

\begin{corollary}
Let $U\in L^\infty(\G) \vnten \mc B(H)$ be an isometric corepresentation.
Then $U$ is unitary.
\end{corollary}
\begin{proof}
As in the proof of Theorem~\ref{thm:main2}, following \cite[Section~4]{kvvn},
the \emph{opposite quantum group} to $\G$ is $\G^\op$, where $L^\infty(\G^\op)
= L^\infty(\G)$ and $\Delta^\op = \sigma\Delta$.  It follows that
corepresentations of $\G$ are co-anti-representations of $\G^\op$, and
vice versa.  The proof follows from the observation that $U^*$ is
thus a coisometric corepresentation of $\G^\op$.
\end{proof}

An application of this result is to the theory of induced corepresentations
in the sense of Kustermans, \cite{kus2}.  In this paper, a theory is
developed allowing one to ``induce'' a corepresentation from a ``smaller''
quantum group to a larger one.  However, in general the resulting
corepresentation is only a coisometry, and not unitary (see
\cite[Notation~5.3]{kus2}), and under a further ``integrability'' condition,
and with an elaborate argument, it is shown that this corepresentation is
unitary, \cite[Proposition~7.5]{kus2}.
Our result shows that a further condition is not necessarily, and
the induced corepresentation is always unitary.

\section{The similarly problem} \label{sec:simmprob}

We recall from \cite[Definition~3.2]{bt} (for example) that $\G$ is
\emph{amenable} if there is a state $m$ on $L^\infty(\G)$ such that
$\ip{m}{(\iota\otimes\omega)\Delta(x)} = \ip{m}{x} \ip{1}{\omega}$
for $x\in L^\infty(\G),\omega\in L^1(\G)$.  Using the unitary antipode, we
see that we could have equivalently used the obvious ``left'' variant of
this definition instead.  We call $m$ a \emph{right invariant state}.

Based on the results of the previous section, we are now a position to state the  ``similarity problem'' for locally compact quantum groups: Suppose $\G$ has the property that every completely bounded homomorphism $\pi:L^1(\G)\rightarrow\mc B(H)$, with $\pi^*$ also extending to a completely bounded homomorphism, is similar to a $*$-homomorphism.  Then is $\G$ necessarily amenable?
%Perhaps it is worth stating what the ``similarity problem'' for locally
%compact quantum groups should be: if $\G$ is such that, whenever
%$\pi:L^1(\G)\rightarrow\mc B(H)$ is a completely bounded homomorphism,
%with $\pi^*$ extending to a completely bounded homomorphism, we have that
%$\pi$ is similar to a $*$-homomorphism, then is $\G$ amenable?
By Theorem~\ref{thm:summ1}, this is equivalent to asking if
having all invertible corepresentations being similar to unitary corepresentations
forces $\G$ to be amenable?  When $\G$ is commutative, this reduces to
the standard conjecture that a unitarisable locally compact group $G$
is amenable.  When $\G$ is cocommutative, \cite{bs} proves this conjecture
in the affirmative-- there is of course the stronger conjecture that actually,
no condition on $\pi^*$ is needed, and we provide more evidence for this
in Section~\ref{sec:fourieralgebra} below.

In the following, we shall show how to use an invariant state to ``average''
an invertible corepresentation to a unitary corepresentation.  This is
well-known in the compact quantum group case (see, for example,
\cite[Theorem~5.2]{woro1}).  The non-compact case is similar, but we
provide the details, as we are unaware of a good reference.

\begin{theorem}\label{thm:main}
Let $\G$ be an amenable locally compact quantum group, and
let $\pi:L^1(\G)\rightarrow\mc B(H)$ be a non-degenerate homomorphism.
The following are equivalent:
\begin{enumerate}
\item\label{thm:main:1} $\pi$ is similar to a $*$-homomorphism;
\item\label{thm:main:2} both $\pi$ and $\check\pi$ extend to completely
bounded (anti-)homomorphisms $L^1(\G)\rightarrow\mc B(H)$.
\end{enumerate}
\end{theorem}
\begin{proof}
That (\ref{thm:main:1}) implies (\ref{thm:main:2}) is 
Proposition~\ref{prop:three}.

Now suppose that (\ref{thm:main:2}) holds, and that $\pi$ is
non-degenerate, so by Theorem~\ref{thm:main2}, $V_\pi$ is invertible.
Let $m\in L^\infty(\G)^*$ be a right invariant state.  Set
\[ T = (m\otimes\iota)(V_\pi^* V_\pi) \in\mc B(H). \]
This means (by definition) that $\ip{T}{\omega} =
\ip{m}{(\iota\otimes\omega)(V_\pi^*V_\pi)}$ for $\omega\in \mc B(H)_*$.

As $V_\pi$ is invertible, so is the positive operator $V_\pi^*V_\pi$, and
so, for some $\epsilon>0$, we have that $V_\pi^*V_\pi \geq \epsilon 1$.
It follows that for any $\omega\in\mc B(H)_*^+$,
\[ \ip{T}{\omega} = \ip{m}{(\iota\otimes\omega)(V_\pi^*V_\pi)}
\geq \ip{m}{(\iota\otimes\omega)(\epsilon 1)} = \epsilon \ip{\omega}{1}. \]
Hence $T\geq \epsilon 1$, and so we may define
\[ V = (1\otimes T^{1/2}) V_\pi (1\otimes T^{-1/2}) \in L^\infty(\G)
\vnten \mc B(H). \]
Then, as $\Delta$ is a unital $*$-homomorphism,
\begin{align*} (\Delta\otimes\iota)(V) &= (1\otimes 1\otimes T^{1/2})
V_{\pi,13} V_{\pi,23} (1\otimes 1\otimes T^{-1/2}) \\
&= (1\otimes 1\otimes T^{1/2}) V_{\pi,13} (1\otimes 1\otimes T^{-1/2})
(1\otimes 1\otimes T^{1/2}) V_{\pi,23} (1\otimes 1\otimes T^{-1/2})
= V_{13} V_{23}. \end{align*}
So $V$ is a corepresentation, say inducing a homomorphism $\phi:L^1(\G)
\rightarrow\mc B(H)$.  For $\omega\in L^1(\G)$, we have that
$\phi(\omega) = (\omega\otimes\iota)(V) = T^{1/2} \pi(\omega) T^{-1/2}$,
and so $\phi$ is similar to $\pi$.

We next show that $V$ is unitary.  For $\omega_1\in L^1(\G)$ and
$\omega_2\in\mc B(H)_*$, we may define $\psi\in\mc B(H)_*$ by
\[ \ip{R}{\psi} = \ip{V_\pi^*(1\otimes R)V_\pi}{\omega_1\otimes\omega_2}
\qquad (R\in\mc B(H)). \]
Then we see that
\begin{align*}
(\iota\otimes\psi)(V_\pi^* V_\pi)
&= (\iota\otimes\omega_1\otimes\omega_2)\big(
   V_{\pi,23}^* V_{\pi,13}^* V_{\pi,13} V_{\pi,23} \big)
= (\iota\otimes\omega_1\otimes\omega_2)
   \big((\Delta\otimes\iota)(V_\pi^*V_\pi)\big).
\end{align*}
Thus
\begin{align*} \ip{T}{\psi} = \ip{m}{(\iota\otimes\psi)(V_\pi^*V_\pi)}
&= \ip{m}{(\iota\otimes\omega_1\otimes\omega_2)
   \big((\Delta\otimes\iota)(V_\pi^*V_\pi)\big)} \\
&= \ip{m}{(\iota\otimes\omega_1)
   \Delta( (\iota\otimes\omega_2)(V_\pi^*V_\pi) )} \\
&= \ip{m}{(\iota\otimes\omega_2)(V_\pi^*V_\pi)} \ip{1}{\omega_1}
= \ip{1\otimes T}{\omega_1\otimes\omega_2}.
\end{align*}
As $\omega_1,\omega_2$ were arbitrary, we conclude that
$V_\pi^*(1\otimes T)V_\pi = 1\otimes T$.  Thus
\begin{align*}
V^*V &= (1\otimes T^{-1/2})V_\pi^* (1\otimes T^{1/2})
(1\otimes T^{1/2}) V_\pi (1\otimes T^{-1/2})
= (1\otimes T^{-1/2})(1\otimes T)(1\otimes T^{-1/2}) = 1,
\end{align*}
as required.

Arguing as in Theorem~\ref{thm:inv_case}, it is now easy to see that $\phi$
is a $*$-homomorphism, because $V$ is unitary.
\end{proof}

We can also deal with the case of degenerate representations of
$L^1(\G)$.  The proof of \cite[Proposition~6]{bs} actually shows the following:

\begin{lemma}\label{lem:degen}
Let $A$ be a $*$-algebra, and let $\pi:A\rightarrow\mc B(H)$
be a homomorphism.  Let $\pi_e:A\rightarrow\mc B(H_e)$ be the subrepresentation
given by restricting to the essential space $H_e$.  Suppose there is a
(not necessarily orthogonal) projection $Q$ from $H$ to $H_e$ such that
$\pi(a)Q = \pi(a)$ for all $a\in A$.  Then $\pi$ is similar to $*$-representation
if and only if $\pi_e$ is similar to a $*$-representation.
\end{lemma}

\begin{theorem}\label{thm:maindeg}
Let $\G$ be an amenable locally compact quantum group, and
let $\pi:L^1(\G)\rightarrow\mc B(H)$ be a homomorphism.
The following are equivalent:
\begin{enumerate}
\item\label{thm:maind:1} $\pi$ is similar to a $*$-representation;
\item\label{thm:maind:2} both $\pi$ and $\check\pi$ extend to completely
bounded (anti-)representations $L^1(\G)\rightarrow\mc B(H)$.
\end{enumerate}
\end{theorem}
\begin{proof}
That (\ref{thm:maind:1}) implies (\ref{thm:maind:2}) is 
again Proposition~\ref{prop:three}.  If (\ref{thm:maind:2}) holds,
then Theorem~\ref{thm:main2} shows that $V_\pi$ restricts to an invertible
operator on $L^2(\G) \otimes H_e$; that is, $V_{\pi_e}$ is invertible.
So Theorem~\ref{thm:main} shows that $\pi_e$ is similar to a $*$-representation.
By Lemma~\ref{lem:degen} it suffices to construct a projection $Q:H
\rightarrow H_e$ with $\pi(\omega) Q = \pi(\omega)$ for $\omega\in L^1(\G)$.

By Theorem~\ref{thm:main2}, $P = V_\pi V_{\check\pi}$ is a projection of
$L^2(\G)\otimes H$ onto $L^2(\G)\otimes H_e$.  The proof actually shows that
$\ker P = \ker V_\pi V_{\check\pi} = (L^2(\G)\otimes\tilde H_e)^\perp
= L^2(\G)\otimes \tilde H_e^\perp$.  Thus $L^2(\G)\otimes H_e \cap
L^2(\G)\otimes \tilde H_e^\perp = \{0\}$, and so $H_e \cap \tilde H_e^\perp
= \{0\}$.  Similarly, $L^2(\G)\otimes H_e + L^2(\G)\otimes \tilde H_e^\perp =
L^2(\G) \otimes H$, and so $H_e + \tilde H_e^\perp = H$.  It follows that there
is a projection $Q$ of $H$ onto $H_e$ with $\ker Q = \tilde H_e^\perp$.

Now, $\alpha \in \tilde H_e^\perp$ if and only if
\[ 0 = (\alpha|\tilde\pi(\omega)\beta) = (\alpha|\pi(\omega^*)^*\beta)
= (\pi(\omega^*)\alpha|\beta) \qquad (\omega\in L^1(\G), \beta\in H). \]
It follows that for $\omega\in L^1(\G), \alpha\in \ker Q$, we have that
$\pi(\omega)\alpha = 0 = \pi(\omega) Q \alpha$.  Clearly $\pi(\omega) Q
= \pi(\omega)$ on $H_e$, and thus it follows that $\pi(\omega)Q = \pi(\omega)$
on $H$, as required.
\end{proof}

\section{Special cases} \label{sec:specialcases}

In this section, we take an approach which is closer in spirit to \cite{bs}.
A cost is that we shall have to assume that $\hat\G$ is coamenable.  Recall that
$\hat\G$ is \emph{coamenable} if $L^1(\hat\G)$ has a bounded
approximate identity.  This is equivalent to a large number of other conditions,
see \cite[Theorem~3.1]{bt}.  In particular, the proof of this reference shows
that $L^1(\hat\G)$ has a contractive approximate identity in this case.  It is
conjectured that the coamenability of $\hat \G$ is equivalent to $\G$ being
amenable, but except when $\G$ is discrete (see \cite{tom}) it is only known that
$\hat\G$ coamenable implies that $\G$ is amenable.  A benefit to the approach of
this section is that in special cases (as for $\hat \G$ being a SIN group in
\cite{bs}) we can show that if $\pi:L^1(\G)\rightarrow\mc B(H)$ is completely
bounded, then $\pi$ is similar to a $*$-homomorphism,
without assumption about $\check\pi$.

\begin{theorem}\label{thm:two}
Let $\G$ be a locally compact quantum group such that $\hat\G$ is coamenable.
Let $\pi:L^1(\G)\rightarrow\mc B(H)$ be a completely bounded homomorphism,
such that $\check\pi$ extends a completely bounded map
$L^1(\G)\rightarrow\mc B(H)$.  Then $\pi$ is similar to a $*$-homomorphism.
\end{theorem}
\begin{proof}
Let $\alpha,\beta\in H$ and $\xi, \eta \in L^2(\G)$.  From
Corollary~\ref{cor:cbmults}, $T^{\tilde\pi}_{\beta,\alpha}
\hat\lambda(\hat \omega_{\xi,\eta}) \in \lambda(L^1(\hat\G))$.  Moreover, an
application of Theorem~\ref{thm:one} to $T^{\tilde\pi}_{\beta,\alpha}$ shows that
\[ T^{\tilde\pi}_{\beta,\alpha} \hat\lambda(\hat\omega_{\xi,\eta})
= \sum_i \hat\lambda\big( \hat\omega_{a_i\xi,S(b_i^*)\eta} \big), \]
where, if $(f_i)$ is an orthonormal basis for $H$,
\[ a_i = T^{\tilde\pi}_{\beta,f_i}, \quad
S(b_i^*) = S\big( (T^{\tilde\pi}_{f_i,\alpha})^* \big)
= T^{\check\pi}_{\alpha,f_i}. \]
In particular,
\begin{align*} \sum_i \|\hat\omega_{a_i\xi,S(b_i^*)\eta}\|
&\leq \Big( \sum_{i}\|a_i\xi\|^2\Big)^{1/2} \Big( \sum_{i}\|S(b_i^*)\eta\|^2\Big)^{1/2} \\ 
&\leq \Big\| \sum_i a_i^*a_i \Big\|^{1/2} \Big\| \sum_i S(b_i^*)^*S(b_i^*)
\Big\|^{1/2} \|\xi\| \|\eta\| \\
&\leq \|\xi\| \|\eta\| \|\alpha\| \|\beta\| \|V_{\tilde\pi}\|
\|V_{\check\pi}\|
= \|\xi\| \|\eta\| \|\alpha\| \|\beta\| \|\pi\|_{cb} \|\check\pi\|_{cb}.
\end{align*}

Let $(\hat\omega_k)$ be a contractive approximate identity for $L^1(\hat\G)$.
It follows that for each $k$, there exists $\hat\omega'_k\in L^1(\hat\G)$ with
\[ (T^\pi_{\alpha,\beta})^* \hat\lambda(\hat\omega_k) =
T^{\tilde\pi}_{\beta,\alpha} \hat\lambda(\hat\omega_k)
= \hat\lambda(\hat\omega'_k) \quad\text{with}\quad
\|\hat\omega'_k\| \leq \|\alpha\| \|\beta\| \|\pi\|_{cb} \|\check\pi\|_{cb}. \]
As $\hat\lambda(L^1(\hat\G))$ is $\sigma$-weakly dense in $L^\infty(\G)$,
it follows that for $\omega\in L^1(\G)$,
\[ \overline{ \ip{T^\pi_{\alpha,\beta}}{\omega} }
= \ip{(T^\pi_{\alpha,\beta})^*}{\omega^*}
= \lim_k \ip{(T^\pi_{\alpha,\beta})^* \hat\lambda(\hat\omega_k)}{\omega^*}
= \lim_k \ip{\hat\lambda(\hat\omega'_k)}{\omega^*}
= \lim_k \ip{\lambda(\omega)^*}{\hat\omega'_k}. \]
The final equality follows by a simple calculation, see
\cite[Lemma~8.10]{dm}.  Thus
\[ |(\pi(\omega)\alpha|\beta)|
= |\ip{T^\pi_{\alpha,\beta}}{\omega}|
\leq \|\alpha\| \|\beta\| \|\pi\|_{cb} \|\check\pi\|_{cb} \|\lambda(\omega)\|. \]
As $\lambda(L^1(\G))$ is dense in $C_0(\hat\G)$, by continuity,
there is a homomorphism $\phi:C_0(\hat\G) \rightarrow \mc B(H)$ with
\[ \|\phi\| \leq \|\pi\|_{cb} \|\check\pi\|_{cb}, \quad
\phi\circ\lambda = \pi. \]

As $\hat\G$ is coamenable, it follows that $\G$ is amenable (see
\cite[Theorem~3.2]{bt}) and that hence $C_0(\hat\G)$ is a nuclear C$^*$-algebra
(see \cite[Theorem~3.3]{bt}).  The similarly problem has an affirmative answer for nuclear
$C^*$-algebras (see \cite[Theorem~4.1]{ch}) so $\phi$ is similar to a
$*$-homomorphism.  As $\lambda$ is a $*$-homomorphism on
$L^1_\sharp(\G)$, it follows that $\pi$ (restricted to $L^1_\sharp(\G)$) is
also similar to a $*$-homomorphism.

The above reference to \cite{ch} needed that $\phi$ (that is, $\pi$)
is non-degenerate.  But as $C_0(\hat\G)$ has a bounded approximate
identity, we can simply follow \cite[Proposition~6]{bs}
to deal with the case when $\pi$ is degenerate.
\end{proof}

The above proof is interesting because in special cases, it allows us
remove the hypothesis that $\check\pi$ is completely bounded (or even
bounded on all of $L^1(\G)$).  This idea was used for SIN groups in
\cite{bs}, the key point being that if $G$ is SIN, then the Plancheral
weight on $VN(G)$ can be approximated by \emph{tracial} states in $A(G)$.
It seems unlikely that many genuinely ``quantum'' groups will have this
property, and so we shall restrict attention to the compact case,
where we have many non-trivial examples.

Recall that a locally compact quantum group $\G$ is \emph{compact}
if $C_0(\G)$ is unital.  Here, the existence of Haar weights follows
from simple axioms, see \cite{woro2}.  We shall be interested in the
case when the Haar state is tracial-- by \cite[Theorem~2.5]{woro2}
this is equivalent to $\G$ being a compact Kac algebra.
See \cite{basp, wa2, wa1} for some genuinely ``quantum'' examples of
compact Kac algebras.

\begin{theorem} \label{thm:Kac}
Let $\G$ be a compact Kac algebra, and let $\pi:L^1(\G) \rightarrow
\mc B(H)$ be completely bounded homomorphism.  Then $\pi$ is similar
to a $*$-homomorphism.
\end{theorem}
\begin{proof}
As the dual $\hat\G$ is discrete, $L^1(\hat\G)$ is unital.  Indeed,
if $\Lambda:L^\infty(\G)\rightarrow L^2(\G)$ is the GNS map, then it
is not hard to show that the state $\hat\omega_{\Lambda(1)}$ is the
unit of $L^1(\hat\G)$.  Following the proof of Theorem~\ref{thm:two}, we wish
to show that $T^{\tilde\pi}_{\beta,\alpha} \hat\lambda(\hat\omega_{\Lambda(1)})
= \hat\lambda(\hat\omega)$ for some $\hat\omega$ whose norm is controlled.
As $S=R$ in this case, we really have to show that
\[ \sum_i \Big\| R\big( (T^{\tilde\pi}_{f_i,\alpha})^* \big)
\Lambda(1) \Big\|^2 
= \sum_i \varphi\Big( R\big( (T^{\tilde\pi}_{f_i,\alpha})^* \big)^*
R\big( (T^{\tilde\pi}_{f_i,\alpha})^* \big) \Big)
= \sum_i \varphi\Big( (T^{\tilde\pi}_{f_i,\alpha})^* T^{\tilde\pi}_{f_i,\alpha} 
\Big) < \infty. \]
where $\varphi$ is the (normal, tracial) Haar state on $L^\infty(\G)$,
which satisfies $\varphi\circ R = \varphi$.  However, as $\varphi$ is a trace,
this sum is
\[ \sum_i \varphi\Big( T^{\tilde\pi}_{f_i,\alpha}
(T^{\tilde\pi}_{f_i,\alpha})^* \Big)
= \sum_i \varphi\Big( (T^{\pi}_{\alpha,f_i})^*
T^{\pi}_{\alpha,f_i} \Big)
= \varphi\big( (\iota\otimes\omega_\alpha)(V_\pi^*V_\pi) \big)
\le \|\pi\|_{cb}^2 \|\alpha \|^2. \]
Then we can just continue as in the proof of Theorem~\ref{thm:two}.
\end{proof}

The above proof could be adapted to a general Kac algebra provided that
we can find a bounded approximate identity $(\hat\omega_{\xi_k,\eta_k})$
for $L^1(\hat\G)$ such that, for each $k$, we have that $\omega_{\eta_k}$
is tracial (by scaling $\eta_k$ and $\xi_k$, we can suppose that $\|\eta_k\|=1$
for each $k$, but we also need at least that $\sup_k \|\xi_k\|<\infty$).
In particular, we can only apply this to $\hat G$ if $G$ is a SIN group,
by using the proof of \cite[Proposition~3.2]{tay} (and so we have reproved
\cite[Theorem~20]{bs}).  However, this argument will not extend to other
classes of groups.

For a general compact quantum group, we can consider the Hopf $*$-algebra
$\mc A\subseteq C_0(\G)$ formed from the matrix coefficients of irreducible
representations.  On this algebra, we have reasonably explicit formulae
for the antipode $S$ and the Haar weight $\varphi$ (see \cite[Section~3.2]{tt}
for example).  A little calculation
shows that if, in reasonable generality, we have something like
\[ \sum_i \varphi\big( S(b_i^*)^* S(b_i^*) \big)
= \sum_i \varphi(b_ib_i^*), \]
then already $\varphi$ is a trace.  So to extend the above result to
general compact quantum groups (if it is true at all!)
probably needs a new idea.

\section{The importance of complete boundedness} \label{sec:importance_cb}

Up to this point, we have restricted our attention to \textit{completely bounded} representations $\pi:L^1(\G) \to \mc B(H)$, and derived some interesting structure results for such representations.  A natural question which arises is: what can be said about bounded representations $\pi$ which are \textit{not} completely bounded, if such representations even exist?   
In \cite{cs} Choi and the 3rd named author answered the existence part of this question by constructing a bounded, but not
completely bounded, representation $\pi:A(G)\rightarrow\mc B(H)$, for any
group $G$ which contains $\mathbb F_2$ as a closed subgroup.  In particular,
for $G=\mathbb F_2$, every completely bounded representation of $A(G)$ is
similar to a $*$-representation (\cite[Theorem~20]{bs}, see also
Theorem~\ref{thm:Kac} above) but there are bounded representations of
$A(G)$ not similar to $*$-representations.

In this section, we will consider free products of compact quantum groups,
and Khintchine inequalities for reduced free products, to construct examples of
non-commutative/non-cocommutative compact quantum groups $\G$ such that
$L^1(\G)$ admits bounded representations $\pi:L^1(\G)\rightarrow\mc B(H)$
which are not similar to $*$-representations.  To be precise, we will show
the following.

\begin{theorem}\label{thm:freeproduct}
Let $\G$ be a non-trivial compact quantum group, and let
$\widetilde{\G} =\ast_{i \in \mathbb N} \G_i$ be the compact quantum group free
product of countably many copies of $\G$ (or let $\widetilde{\G} = \mathbb T*\G$).  Then there exists a bounded representation $\pi:L^1(\widetilde{\G}) \rightarrow \mc B(H)$ which is not completely bounded.  In particular, $\pi$ is not similar to a $*$-representation.
\end{theorem}

For completely bounded representations $\pi:L^1(\G) \to \mc B(H)$ (with $\pi^*$  also completely bounded), a key result in the preceding sections was Corollary   \ref{cor:cbmults}, which showed that each co-efficient operator $T^{\tilde\pi}_{\alpha, \beta} \in L^\infty(\G)$ ($\alpha,\beta \in H$) induces a completely bounded multiplier of the dual convolution algebra $L^1(\hat \G)$.  The following corollary shows that for non-completely bounded representations, there is no hope of generalising Corollary \ref{cor:cbmults}. 

\begin{corollary} \label{cor:freeproduct}
Let $\widetilde{\G}$ and $\pi:L^1(\widetilde{\G}) \rightarrow \mc B(H)$ be as in the statement and proof of Theorem \ref{thm:freeproduct}.  Then there exist vectors
$\alpha, \beta \in H$ with the property that the map $\hat \lambda(\hat \omega) \mapsto T^{\tilde\pi}_{\alpha,\beta}\hat \lambda(\hat \omega)$ $(\hat \omega \in L^1(\widehat{\widetilde{\G}}))$ does not induce a bounded left multiplier of $L^1(\widehat{\widetilde{\G}})$.
\end{corollary}

\begin{proof}[Proof of Corollary \ref{cor:freeproduct}]
We use an argument similar to the proof of
\cite[Theorem 2.3]{dh}.  Suppose, to get a contradiction, that each co-efficient
$T^{\tilde\pi}_{\alpha,\beta}$ induces a multiplier of $L^1(\widehat{\widetilde{\G}})$, as
in Corollary~\ref{cor:cbmults}.
Then for each $\alpha, \beta \in H$, as $L^1(\widehat{\widetilde{\G}})$ is unital,
there exists a unique $\hat\omega^{\tilde\pi}_{\alpha, \beta} \in L^1(\widehat{\widetilde{\G}})$ such that
$T^{\tilde\pi}_{\alpha, \beta} =\hat\lambda(\hat\omega^{\tilde\pi}_{\alpha, \beta})$.  Define a
sesquilinear map $Q:H \times H \to L^1(\widehat{\widetilde{\G}})$ by setting
$Q(\alpha, \beta) = \hat\omega^{\tilde\pi}_{\alpha, \beta}$.  We claim that $Q$ is separately
continuous.  To see this, fix $\beta \in H$ and suppose
$\lim_{n \to \infty}\|\alpha - \alpha_n\| = 0$ and
$\lim_{n\to \infty}Q(\alpha_n,\beta) = \hat \omega \in L^1(\widehat{\widetilde{\G}})$.
Then for any $\omega \in L^1(\widetilde{\G})$, 
\begin{align*}
\ip{\hat\lambda(\hat\omega)}{\omega} &=
\lim_{n\rightarrow\infty} \ip{\hat\lambda(Q(\alpha_n,\beta))}{\omega}
= \lim_{n\rightarrow\infty} \ip{T^{\tilde\pi}_{\alpha_n,\beta}}{\omega}
= \lim_{n\rightarrow\infty} \big({\tilde\pi}(\omega) \alpha_n \big| \beta \big) \\
&= \big({\tilde\pi}(\omega) \alpha \big| \beta \big)
= \ip{T^{\tilde\pi}_{\alpha,\beta}}{\omega}
= \ip{\hat\lambda(Q(\alpha,\beta))}{\omega}.
\end{align*}
Therefore $\hat\lambda(\hat\omega) = \hat\lambda(Q(\alpha,\beta))$, giving $\hat\omega = Q(\alpha,\beta)$,
and the closed graph theorem implies that $Q$ is continuous in the first variable.
The same argument applies to the second variable, proving the claim.  

Since any separately continuous sesquilinear map is bounded
(see for example \cite[Theorem 2.17]{ru}), there is a constant $D>0$ such that
$\|\hat\omega^{\tilde\pi}_{\alpha,\beta}\|_{L^1(\widehat{\widetilde{\G}})} \le D\|\alpha\|\|\beta\|$
for all $\alpha, \beta \in H$.  As a consequence, we have
\begin{align*}
|\big( {\tilde\pi}(\omega)\alpha|\beta\big)| &=
| \ip{\hat\lambda(\hat \omega^{\tilde\pi}_{\alpha, \beta})}{\omega} |
= |\ip{\lambda(\omega^*)^*}{\hat \omega^{\tilde\pi}_{\alpha,\beta}}| \\
&\leq \| \hat \omega^{\tilde\pi}_{\alpha,\beta}\| \|\lambda(\omega^*)^*\|
\leq D \|\alpha\| \|\beta\| \|\lambda(\omega^*)\| \qquad (\omega \in L^1(\widetilde{\G})),
\end{align*}
showing that ${\tilde\pi}$ extends to the bounded anti-representation
$\sigma:C_0(\widehat{\widetilde{\G}})\to \mc B(H)$ determined by
$\sigma \circ\lambda = \tilde\pi$.  Since $\widetilde{\G}$ is compact, $C_0(\widehat{\widetilde{\G}})$ is nuclear, and it follows that $\sigma$ is automatically completely bounded.  But then $\|\pi\|_{cb} = \|\tilde\pi\|_{cb} = \|\sigma \circ \lambda\|_{cb} \le \|\sigma\|_{cb}< \infty$, contradicting the fact that $\pi$ is not completely bounded.
\end{proof}

\subsection{Reduced free products}

Before proceeding with the proof of Theorem \ref{thm:freeproduct}, let us recall the construction of the reduced free product of a collection
of C$^*$-algebras equipped with distinguished states.  For further details
see \cite{av}, \cite{vo} or the introduction to \cite{rx}.  Let $I$ be an
index set, and for each $i\in I$ let $A_i$ be a C$^*$-algebra (assumed unital),
and let $\varphi_i$ be a state on $A_i$ with faithful GNS construction
$(\pi_i, H_i, \xi_i)$.  Let $D_I$ be the collection of all ``words''
in $I$, that is, tuples $(i_1,\cdots,i_n)$ with $i_j \not= i_{j+1}$ for each
$j$; we allow the empty word $\emptyset$.
Set $H_i^0 = \xi_i^\perp \subseteq H_i$, and let $H$ be the Hilbert space
direct sum of $H_{i_1}^0 \otimes \cdots \otimes H_{i_n}^0$ as $(i_j)$ varies
through $D_I$; we interpret the space for $\emptyset\in D_I$ as being a copy of
$\mathbb C$.  Thus $H$ is some sort of generalised Fock space construction.
Let $\Omega\in H$ be the vector $1\in\mathbb C$.

For each $i\in I$ there is an obvious isomorphism
\begin{align*} H &\cong (\mathbb C \oplus H_i^0) \otimes
\Big( \bigoplus \big\{ H_{i_1}^0 \otimes\cdots\otimes H_{i_n}^0 :
(i_j)\in D_I, i_1\not=i \big\} \Big) \\
&\cong H_i \otimes
\Big( \bigoplus \big\{ H_{i_1}^0 \otimes\cdots\otimes H_{i_n}^0 :
(i_j)\in D_I, i_1\not=i \big\} \Big). \end{align*}
Let $U_i$ be the unitary implementing this isomorphism, and define
$\hat\pi_i:A_i\rightarrow\mc B(H); a\mapsto U_i^*(\pi_i(a)\otimes 1)U_i$,
a faithful $*$-representation of $A_i$ on $H$.  A description of
$U_i$ is given in \cite{av}, but let us just note that if $a_i\in A_i$ with
$\varphi_i(a_i)=0$, then $\hat\pi_i(a_i)\Omega = a_i\xi_i \in H_i^0$.

Let $\mc A$ be the algebraic free product of the $(A_i)$,
amalgamated over units.  Then $\pi = \ast_{i\in I} \hat\pi_i$ is a faithful
$*$-representation of $\mc A$ on $\mc B(H)$.  By definition, the reduced
free product $\ast_{i\in A} A_i$ is the closure of $\pi(\mc A)$.  The vector
state induced by $\Omega$ is the free product of the states, denoted
$\ast_{i\in I} \varphi_i$.  Notice that $(\pi,H,\Omega)$ is the GNS
construction for $\ast_{i\in I} \varphi_i$.
Denote by $A$ the full free product, which is
the completion of $\mc A$ under the norm given by considering all
$*$-representations of $A_i$ on a common Hilbert space.

If each $A_i$ were a von Neumann algebra, with normal states $\varphi_i$,
then the (reduced) von Neumann algebraic free product is simply the
weak operator closure of $\ast_{i\in I} A_i$ in $\mc B(H)$.

In \cite{wa1} Wang studied free products of compact quantum groups
(see also \cite[Section~6.3]{tt}).
If for each $i$ we have a compact quantum group $\G_i$, then we form
$\G = \ast_{i\in I} \G_i$ as follows.  Let $C_0(\G)$ be the
reduced free product of $(C_0(\G_i),\varphi_i)$, where $\varphi_i$ is
the Haar state.  For each $i$ let $\iota_i:C_0(\G_i) \rightarrow A$ be the
inclusion, and let $\Delta_i$ be the coproduct on $C_0(\G_i)$.
Set $\rho_i = (\iota_i\otimes\iota_i)\Delta_i: C_0(\G_i) \rightarrow
A\otimes A$.  By the universal property of $A$, there is a
$*$-homomorphism $\Delta: A\rightarrow A\otimes A$ such that $\Delta\iota_i
= \rho_i$ for each $i$.  It is easy to see that $\Delta$ is
coassociative, and by using the corepresentation theory of compact
quantum groups, one can verify the density conditions to show that
$(A,\Delta)$ is a compact quantum group.

Then \cite[Theorem~3.8]{wa1} shows that the Haar state on $(A,\Delta)$ is
just the free product of the Haar states on each $C_0(\G_i)$.  It follows
that the reduced version of $A$ is just $C_0(\G)$, and thus
$\Delta$ drops to a coproduct on $C_0(\G)$.

We require the following non-commutative Khintchine inequality 
describing the operator space structure of the linear span of a family
of centred freely independent operators.  See \cite[Proposition~7.4]{ju}
or the introduction to \cite{rx}.  In fact, we shall only use the $k=1$
case, which is attributed to Voiculescu \cite{Voi}.

\begin{theorem}\label{thm:Voic_inequality}
Let $I$ be an index set, and for each $i \in I$, let $A_i$ be a von Neumann
algebra with faithful normal state $\varphi_i$.  Let
$(A,\varphi) = (\ast_{i \in I}A_i, \ast_{i \in I} \varphi_i)$ be their
von Neumann algebraic free product.  Then for any $x_i \in A_i$ with
$\varphi_i(x_i)=0$, and any finitely supported family
$\{a_i\}_{i \in I} \in M_k(\C)$ ($k \in \mathbb N$), we have
\begin{align*}
&\max \Big\{\max_{i \in I}\|a_i \otimes x_i\|_{M_k(\C) \otimes A},
\Big\|\sum_{i \in I} a_i^*a_i \varphi_i(x_i^*x_i)\Big\|_{M_k(\C)}^{1/2},
\Big\|\sum_{i \in I} a_ia_i^* \varphi_i(x_ix_i^*)\Big\|_{M_k(\C)}^{1/2} \Big\} \\
&\leq \Big\|\sum_{i \in I} a_i \otimes x_i\Big\|_{M_k(\C) \otimes A} \\
&\leq 3 \max \Big\{\max_{i \in I}\|a_i \otimes x_i\|_{M_k(\C) \otimes A},
\Big\|\sum_{i \in I} a_i^*a_i \varphi_i(x_i^*x_i)\Big\|_{M_k(\C)}^{1/2},
\Big\|\sum_{i \in I} a_ia_i^* \varphi_i(x_ix_i^*)\Big\|_{M_k(\C)}^{1/2} \Big\}.
\end{align*}
\end{theorem}

Now let $\G$ be a non-trivial compact quantum group, set $\G_i=\G$ for each
$i\in\mathbb N$, and let $\widetilde{\G} = \ast_{i \in \mathbb N}\G_i$ be the free product.  (We will show how to modify the following arguments to address the case of $\widetilde{\G} = \mathbb T*\G$ in Remark \ref{rem:torus}.) Fix once and for all a non-trivial irreducible
unitary representation $U \in L^\infty(\G) \otimes M_{d}(\C)$ and for each
$i \in \mathbb N$, let $U^i \in L^\infty(\widetilde{\G}) \otimes M_d(\C)$ be the
irreducible unitary representation of $\widetilde{\G}$ induced by the inclusion
$L^\infty(\G_i) \hookrightarrow L^\infty(\widetilde{\G})$.  Note that the $U^i$'s
are pairwise inequivalent representations.  Finally, let
$L^\infty_i = \{(\iota \otimes \rho) U^i : \rho \in M_d(\C)^*\}$
be the co-efficient space of $U^i$ and let $X$ be the weak$^*$-closed linear
span of all the $L^\infty_i$'s.

\begin{lemma}\label{lem:norm_equivalence}
There is a constant $C > 0$ such that
\[ \| x\Omega \|_{L^2(\widetilde{\G})} \leq
\|x\|_{L^\infty(\widetilde{\G})} \leq
C\| x\Omega \|_{L^2(\widetilde{\G})} \qquad (x \in X).\] 
\end{lemma}
\begin{proof}
The lower bound is immediate.  To get the upper bound, fix $x \in X$ and write
$x$ as an $L^2$-convergent series $x=\sum_{i \in \mathbb N}x_i \in X$ where each
$x_i$ is a member of $L^\infty_i$.  As $U$ is non-trivial, the Haar
state annihilates all coefficients of $U$, and so we can apply the
$k=1$ version of Theorem~\ref{thm:Voic_inequality} to see that
\begin{align*}\|x\|_{L^\infty(\widetilde{\G})} &\leq
3 \max\Big\{ \max_i \|x_i\|_{L^\infty(\widetilde{\G})},
  \Big(\sum_{i \in \mathbb N} \varphi_i(x_i^*x_i)\Big)^{1/2},
  \Big(\sum_{i \in \mathbb N} \varphi_i(x_ix_i^*)\Big)^{1/2}\Big\} \\
&= 3 \max\Big\{ \max_i \|x_i\|_{L^\infty_i},
\| x\Omega \|_{L^2(\widetilde{\G})},
\Big(\sum_{i \in \mathbb N} \| x_i^*\xi_i \|_{L^2(\G_i)}^2\Big)^{1/2}\Big\}.
\end{align*}
Since the spaces $L^\infty_i$ are all finite dimensional and
isometrically isomorphic, there exist constants $C_1, C_2 > 0$ such that
\[ \|x_i\|_{L^\infty_i} \leq C_1 \| x_i\xi_i \|_{L^2(\G_i)}
\quad \text{and} \quad
\| x_i^*\xi_i \|_{L^2(\G_i)} \leq C_2 \| x_i\xi_i \|_{L^2(\G_i)}
\qquad (i \in \mathbb N).\]
The lemma now follows by taking $C = 3 \max\{C_1, C_2\}$. 
\end{proof}

\begin{lemma}
The pre-annihilator ${}^\perp X$ is a closed two-sided ideal in
$L^1(\widetilde{\G})$.
\end{lemma}
\begin{proof}
By linearity, it is enough to show that if $\omega\in{}^\perp X$ and
$x = (\iota \otimes \rho) U^i$ for some $i\in\mathbb N$
and $\rho\in M_d(\C)^*$, then $\ip{x}{\omega\omega'} = 0 = 
\ip{x}{\omega'\omega}$ for $\omega'\in L^1(\widetilde{\G})$.  However,
\[ \ip{x}{\omega\omega'} = \ip{\Delta(x)}{\omega\otimes\omega'}
= \ip{U^i_{13} U^i_{23}}{\omega\otimes\omega'\otimes\rho}
= \ip{U^i}{\omega\otimes (\omega'\otimes\iota)(U^i)\rho} = 0, \]
as $\rho'=(\omega'\otimes\iota)(U^i)\rho\in M_d(\C)^*$ and
$(\iota\otimes\rho')(U^i)\in X$.  Similarly, $\ip{x}{\omega'\omega}=0$.
\end{proof}

Consider now the space $\ell^2(\mathbb N,M_d)$, which is a Banach algebra
under pointwise operations.

\begin{proposition}\label{prop:mikeone}
The map
\[ \phi: L^1(\widetilde{\G}) \rightarrow \ell^2(\mathbb N,M_d); \quad
\omega \mapsto \big( (\omega\otimes\iota)(U^i) \big)_{i\in\mathbb N} \]
is well-defined, bounded, and is an algebra homomorphism.  Furthermore,
$\phi$ drops to give a isomorphism between $L^1(\widetilde{\G})/{}^\perp X$
and $\ell^2(\mathbb N,M_d)$.
\end{proposition}
\begin{proof}
The dual space of $\ell^2(\mathbb N,M_d)$ is $\ell^2(\mathbb N,M_d^*)$.  Let
$\rho=(\rho_i) \in \ell^2(\mathbb N,M_d^*)$, and let
$x_i=(\iota\otimes\rho_i)(U^i) \in L^\infty_i$.  For $n\in\mathbb N$,
let $x=\sum_{i=1}^n x_i \in X$.
Observe that $x\Omega = \sum_{i=1}^n x_i\Omega = \sum_{i=1}^n x_i\xi_i$
as $\varphi_i(x_i)=0$ for each $i$.  As the $x_i\xi_i$ are pairwise
orthogonal in the Fock space $H$, it follows that
\[ \|x\Omega\|^2 = \sum_{i=1}^n \|x_i\xi_i\|^2 \leq \sum_{i=1}^n \|x_i\|^2
\leq \sum_{i=1}^n \|\rho_i\|^2 \leq \|\rho\|^2. \]
Letting $C$ be the constant from Lemma~\ref{lem:norm_equivalence},
we see that
\begin{align*}
\Big| \sum_{i=1}^n \ip{\rho_i}{(\omega\otimes\iota)(U^i)} \Big| &=
\Big| \sum_{i=1}^n \ip{x_i}{\omega} \Big| = \big| \ip{x}{\omega} \big| \\
&\leq \|\omega\| \|x\|
\leq C\|\omega\| \|x\Omega\|
\leq C\|\omega\| \|\rho\|. \end{align*}
As $n$ and $\rho$ were arbitrary, it follows that $\phi$ is well-defined,
and $\|\phi\|\leq C$.

For each $i$, as $U^i$ is a corepresentation, the map $L^1(\widetilde{\G})
\rightarrow M_d(\C); \omega\mapsto (\omega\otimes\iota)(U^i)$ is a homomorphism.
It follows that $\phi$ is also a homomorphism.  Notice that $\phi(\omega)=0$
if and only if $(\omega\otimes\iota)(U^i)=0$ for each $i$, if and only if
$\omega \in {}^\perp L^\infty_i$ for each $i$, if and only if $\omega
\in {}^\perp X$.  So $\phi$ drops to give an injective map
$L^1(\widetilde{\G})/{}^\perp X \rightarrow \ell^2(\mathbb N,M_d)$.

Notice that $\phi^*:\ell^2(\mathbb N,M_d^*) \rightarrow ({}^\perp X)^\perp = X$
is the map $(\rho_i) \mapsto \sum_i (\iota\otimes\rho_i)(U^i)$.
As $U$ is irreducible, the map $M_d^*
\rightarrow L^\infty(\G), \rho\mapsto (\iota\otimes\rho)(U)$ is injective,
and so also the map $M_d^* \rightarrow L^2(\G), \rho\mapsto
\Lambda( (\iota\otimes\rho)(U) )$ is injective.  Thus there is a constant $C'$
such that $\|\rho\| \leq C'\| (\iota\otimes\rho)(U)\xi_i \|$ for all
$\rho\in M_d^*, i\in\mathbb N$.  Then, for
$(\rho_i)\in\ell^2(\mathbb N,M_d^*)$,
\begin{align*}
\| (\rho_i) \|^2 &\leq {C'}^2 \sum_i \|(\iota\otimes\rho_i)(U^i)\xi_i \|^2
= {C'}^2 \Big\| \sum_i (\iota\otimes\rho_i)(U^i) \Omega \Big\|^2 \\
&\leq {C'}^2 \Big\| \sum_i (\iota\otimes\rho_i)(U^i) \Big\|^2
= {C'}^2 \big\| \phi^*((\rho_i)) \big\|^2. \end{align*}
Thus $\phi^*$ is bounded below, which implies that $\phi$ is a surjection,
and so the proof is complete.
\end{proof}

\subsection{Proof of the main theorem}

The final part of the construction is to find a bounded-below homomorphism
$\theta : \ell^2(\mathbb N, M_d) \rightarrow \mc B(H)$ for a suitable
Hilbert space $H$.  Given $\xi,\eta\in H$, denote by $\theta_{\xi,\eta}$
the rank-one operator $H\rightarrow H; \gamma \mapsto (\gamma|\eta)\xi$.
Recall from \cite[Section~3.1]{cs} that if we denote by $(\delta_i)$ the
canonical orthonormal basis of $\ell^2(\mathbb N_0)$ or $\ell^2(\mathbb N)$,
and define $\theta_0:\ell^2(\mathbb N)\rightarrow\mc B(\ell^2(\mathbb N_0))$
by $\theta_0(\delta_i) = \theta_{\delta_i,\delta_i+\delta_0}$, then
$\theta_0$ extends by linearity and continuity to a bounded below algebra
homomorphism.

Set $H = \mathbb C^d\otimes \ell^2(\mathbb N_0)$, and regard
$\ell^2(\mathbb N, M_d)$ as $M_d \otimes \ell^2(\mathbb N)$.
Now define $\theta = (\iota\otimes\theta_0): \ell^2(\mathbb N, M_d)
\rightarrow M_d \otimes \mc B(\ell^2(\mathbb N_0)) \cong \mc B(H)$,
an algebra homomorphism.
As $M_d$ is finite-dimensional, clearly $\theta$ is bounded, and bounded below.
We can now prove the main result of this section.

\begin{proof}[Proof of Theorem~\ref{thm:freeproduct}]
Form $X$ as above, and the algebra isomorphism $\phi:L^1(\widetilde{\G})/{}^\perp X
\rightarrow \ell^2(\mathbb N, M_d)$.  Let $q:L^1(\widetilde{\G}) \rightarrow
L^1(\widetilde{\G})/{}^\perp X$ be the quotient map.  Let $\pi=\theta \circ \phi
\circ q$, so that $\pi$ is a bounded representation of $L^1(\widetilde{\G})$.

Towards a contradiction, suppose that $\pi$ is completely bounded, and
hence associated to a corepresentation $V_\pi\in L^\infty(\widetilde{\G})
\vnten\mc B(H)$.  Set
\[ U = \sum_{i\in\mathbb N} U^i \otimes \theta_{\delta_i,\delta_i}
\in L^\infty(\widetilde{\G})\vnten M_d \vnten \mc B(\ell^2(\mathbb N_0))
\cong L^\infty(\widetilde{\G})\vnten \mc B(H), \]
where the sum converges $\sigma$-weakly (think of $U$ as being a
block diagonal matrix, with diagonal entries $U^i$, which are all unitary).
For $\omega\in L^1(\widetilde{\G})$, $\xi\in\mathbb C^d$ and $k\in\mathbb N_0$,
we have that
\[  \pi(\omega)(\xi\otimes\delta_k)
= \delta_{k,0} (\omega\otimes\iota)(U^i)(\xi) \otimes \delta_i
+ \delta_{k,i} (\omega\otimes\iota)(U^k)(\xi) \otimes \delta_k. \]
Consequently,
\[ (\omega\otimes\iota)(V_\pi-U) (\xi\otimes\delta_k)
= \pi(\omega)(\xi\otimes\delta_k)
- (\omega\otimes\iota)(U^k)(\xi) \otimes \delta_k
= \delta_{k,0} (\omega\otimes\iota)(U^i)(\xi) \otimes \delta_i. \]
If we now set
\[ x_i = (\iota\otimes\omega_{\delta_0,\delta_i})(V_\pi-U)
\in L^\infty(\widetilde{\G})\vnten M_d, \]
then $\sum_i x_i^*x_i$ converges $\sigma$-weakly to 
$(\iota\otimes\omega_{\delta_0,\delta_0})((V_\pi-U)^*(V_\pi-U))$.
However, for $\omega\in L^1(\widetilde{\G})$ and $\xi\in\mathbb C^d$, also
\[ \sum_i (\omega\otimes\iota)(x_i)(\xi) \otimes \delta_i
= \sum_i (\omega\otimes\iota)(V_\pi-U) (\xi\otimes\delta_0)
= \sum_i (\omega\otimes\iota)(U^i)(\xi) \otimes \delta_i. \]
It follows that $x_i = U^i$ for all $i$, and hence each $x_i$ is unitary,
which contradicts $\sum_i x_i^*x_i$ converging.
We conclude that $\pi$ is not completely bounded, and hence cannot be
similar to a $*$-representation.
\end{proof}

Let us make a remark about $\pi^*$.
As each $U^i$ is a unitary corepresentation, we have that
$(\omega^\sharp\otimes\iota)(U^i) = (\omega\otimes\iota)(U^i)^*$ for each
$\omega\in L^1_\sharp(\widetilde{\G})$.  Thus if we give $\ell^2(\mathbb N,M_d)$ the
pointwise $*$ operation, it follows that
$\phi\circ q: L^1(\widetilde{\G}) \rightarrow \ell^2(\mathbb N,M_d)$ is a
$*$-homomorphism.  Thus
\[ \pi^*(\omega) = \pi(\omega^\sharp)^* = \theta\big( \phi(q(\omega))^* \big)^*
\qquad (\omega\in L^1_\sharp(\widetilde{\G})). \]
It follows that $\pi^*$ extends to a bounded homomorphism on $L^1(\widetilde{\G})$.
It is easy to compute explicitly what $\theta( a^* )^*$ is, for each
$a\in\ell^2(\mathbb N,M_d)$, and then adapting the argument in the previous
proof will show that $\pi^*$ is also not completely bounded.

\begin{remark} \label{rem:torus}
We now briefly address the case where $\widetilde{\G} = \mathbb T*\G$ in Theorem \ref{thm:freeproduct}.  Let $z$ denote the canonical Haar unitary generator of 
$L^\infty(\mathbb T)$, and let $U \in L^\infty(\G) \otimes M_d(\C)$ be our fixed non-trivial irreducible unitary representation of $\G$.  For each $i \in \mathbb N$, consider the tensor product unitary representation of $\widetilde{\G}$ given by $V^i = z^i \boxtimes U\boxtimes z^{-i} \in L^\infty(\widetilde{\G}) \otimes M_d(\C)$ (where $z^i$ is viewed as a one-dimensional representation of $\mathbb T$).  It follows from \cite{wa1} that the representations $\{V^i\}_{i \in \mathbb N}$ are pairwise inequivalent and irreducible.  Our claim is that the above proof for $\widetilde{\G} = \ast_{i \in \mathbb N}\G_i$ goes through unchanged with the family $\{U^i\}_{i \in \mathbb N}$ replaced by $\{V^i\}_{i \in \mathbb N}$. 

To see this, observe that the only facts that we used about $\widetilde{\G}$ and $\{U^i\}_{i \in \mathbb N}$ are: (1) each representation $U^i$ is $d$-dimensional, (2) the elements of the co-efficient spaces $L^\infty_i$ associated to $U^i$ are centred with respect to the Haar state, and (3) the von Neumann algebras $\mc A_i$ generated by $L^\infty_i$ are pairwise freely independent in $(L^\infty(\widetilde{\G}), \varphi)$.  For the family $\{V^i\}_{i \in \mathbb N}$,  conditions (1) and (2) are automatically satisfied.  To see that condition (3) is also satisfied, note that if $\mc A_0 \subseteq L^\infty(\G) \hookrightarrow L^\infty(\widetilde{\G})$ denotes the von Neumann algebra generated by the  
co-efficient space of $U$, then $\mc A_i = z^i\mc A_0z^{-i}$ for each $i \in \mathbb N$.  Using this fact, condition (3) is easily seen to be a simple consequence of the definition of free independence (see \cite[Definition 5.3]{nisp}, for example) and the fact that $z$ is $\ast$-free from $\mc A_0$.  \end{remark}

\section{For the Fourier algebra} \label{sec:fourieralgebra}

In this section, we collect some further special cases for the
Fourier algebra.  These add further weight to the conjecture that every
completely bounded representation $\pi:A(G)\rightarrow \mathcal B(H)$ is
similar to a $*$-representation.  The main results of this section are as follows.

\begin{theorem}\label{T:cc rep-amen group}
Let $G$ be an amenable locally compact group, and let $\pi : A(G) \to \mathcal B(H)$
be a completely \emph{contractive} representation. Then $\pi$ is a $*$-representation.
\end{theorem}

\begin{theorem}\label{T:cb rep-amen open SIN sub}
Let $G$ be an amenable locally compact group which contains an open SIN subgroup
$K \subseteq G$. Then every completely bounded representation
$\pi:A(G)\rightarrow\mc \mathcal B(H)$ is similar to a $*$-representation.
\end{theorem}

The proofs of the above results rely on the fact that when $G$ is an amenable
locally compact group, $A(G)$ is always a $1$-operator amenable completely contractive
Banach algebra, see \cite{Ruan}.

\begin{definition}\label{D:vertual Dia-Haagerup}
Let $A$ be completely contractive Banach algebra, and denote by
$m:A \widehat{\otimes}A \to A$ the multiplication map.  A net $\{d_\alpha \}$ in
$A \widehat{\otimes} A$ is called a {\it bounded approximate diagonal} for $A$ if
$\{d_\alpha \}$ is bounded and for every $a\in A$,
\[ a\cdot d_\alpha - d_\alpha \cdot a \rightarrow 0 \quad\text{and}\quad
am(d_\alpha) \rightarrow 0, \]
where $A \proten A$ has the usual $A$-bimodule structure,
\[ a \cdot (b \otimes c)=ab \otimes c, \quad
(b \otimes c) \cdot a =b \otimes ca \qquad
(a,b,c \in A). \]
An element $M \in (A\widehat{\otimes} A)^{**}$ is called a \emph{virtual diagonal}
provided that
\[ a\cdot M=M\cdot a, \quad am^{**}(M)=a \qquad (a\in A). \]
Ruan proved in \cite{Ruan} that the $a$ has a bounded approximate diagonal
in $A\proten A$ if and only if it has a virtual diagonal in
$(A\widehat{\otimes} A)^{**}$ (and, in turn, these are equivalent to $A$ being
\emph{operator amenable}).

Now let $L >0$. $A$ is called \emph{$L$-operator amenable} if it has a virtual
diagonal $M\in (A\widehat{\otimes} A)^{**}$ such that $\|M\|\leq L$. The
\emph{operator amenability constant} of $A$ is
$N:=\inf \{ L : A\text{ is L-operator amenable} \}$.
By Alaoglu's theorem, the operator amenability constant is attained, that is,
$A$ has a $N$-virtual diagonal in $(A\widehat{\otimes} A)^{**}$.
Since any virtual diagonal has norm at least 1, the operator amenability constant
is at least 1.
\end{definition}

\begin{proof}[Proof of Theorem~\ref{T:cc rep-amen group}:]
Since $G$ is amenable, $A(G)$ has a contractive approximate identity.
Hence it follows from \cite[Remark 7]{bs} (see also Lemma~\ref{lem:degen})
that $\pi$ is a $*$-representation if and only if $\pi_e$ is a $*$-representation,
where $\pi_e = \pi|_{H_e}$ is the essential part of $\pi$. 
So, without loss of generality, we can assume that $\pi$ is non-degenerate.

Put $A=\overline{\pi(A(G))}$. Since $A(G)$ is 1-operator amenable and $\pi$ is a
complete contraction, it follows that $A$ is also 1-operator amenable.  
It follows from
\cite[Theorem 7.4.18(ii)]{BL} that $A$ is a (commutative) C$^*$-algebra.
(Note that the proof of  \cite[Theorem 7.4.18(ii)]{BL} is for \emph{unital} operator
algebras but it can be easily modified to apply to our case as well).
We now follow the proof of \cite[Corollary 9]{bs}.
Indeed, the adjoint operation on a commutative C$^*$-algebra
(in particular, on  $A$) is a complete isometry.  Recall that $\pi^*$ is defined by
\[ \pi^*(u) = \pi(\overline{u})^* \qquad (u\in A(G)). \]
Now, the map $u\mapsto\overline{u}$ is an anti-linear complete isometry,
and thus $\|\pi^*\|_{cb} = \|\pi\|_{cb}$.  Hence by \cite[Theorem 8]{bs},
there is an invertible operator $T\in \mathcal B(H)$ and a
$*$-homomorphism $\sigma : C_0(G) \to \mathcal B(H)$ such that
\[ \pi(u)=T^{-1}\sigma(u)T \qquad (u\in A(G)). \]
Moreover, $T$ can be chosen so that
\[ 1 \leq \|T\|\|T^{-1}\| \leq \|\pi\|^2_{cb} \|\pi^*\|^2_{cb} \leq 1. \]
So $ \|T\| \|T^{-1}\| = 1$, and so, by replacing $T$ with $T/\|T\|$,
we can assume that $\|T\|=\|T^{-1}\|=1$ which means that $T$ is unitary.
Hence
\[ \pi(\cdot)=T^{-1}\sigma(\cdot)T=T^*\sigma(\cdot)T \]
is a $*$-homomorphism, as claimed.
\end{proof}

We now turn to the proof of Theorem~\ref{T:cb rep-amen open SIN sub}.
For a locally compact group $G$ and an open subgroup $K$ of $G$, let
$G=\dot{\bigcup}_{x\in I} xK$ denote the decomposition of $G$ to distinct
left cosets of $K$ (i.e. $xK\cap yK=\emptyset$ if $x\neq y$). For every
element $u\in A(G)$ and $x\in I$,  write
\[ u_x=u\chi_{xK}, \]
where $\chi_{xK}$ is the characteristic function of the coset $xK$.
Since $K$ is open, each $\chi_{xK}$ is a norm-one idempotent in the
Fourier-Stieltjes algebra $B(G) = C^*(G)^*$ \cite[Proposition 2.31]{eymard}.
Since $A(G)$ is a closed ideal in $B(G)$, we conclude that $u_x \in A(G)$. We let
\[ A(xK)=A(G)\chi_{xK}. \]
Since $K$ is open, the canonical embedding of the Fourier algebra $A(K)$ into $A(G)$
(that is, extending functions by zero outside of $K$) is completely isometric,
allowing us to identify $A(K)$ with its image $A(eK)$ unambiguously.
In what follows, we shall consider the translation operators
$L_x: A(K) \to A(xK)$ defined by
\[ (L_xu)(y) = u(x^{-1}y), \qquad (u \in A(K)). \]
Since left translation on $A(G)$ is completely isometric,
$L_x: A(K) \to A(xK)$ is always a completely isometric algebra isomorphism.
Finally, for any homomorphism $\pi : A(G) \to \mc B(H)$, let
$\pi_x$ ($x\in I)$ denote the restriction of $\pi$ to the ideal $A(xK)$.

\begin{proof}[Proof of Theorem~\ref{T:cb rep-amen open SIN sub}:]
Since $G$ is amenable, $A(G)$ has a contractive approximate identity, and we
may again assume that $\pi$ is non-degenerate.

As above, we have that $G=\dot{\bigcup}_{x\in I} xK$. For every $x\in
I$, the mapping $\pi_x\circ L_x:A(K) \to \mathcal B(H)$ defines a completely
bounded homomorphism of $A(K)$ on $H$, and so by \cite[Theorem 20]{bs},
$\pi_x\circ L_x$ extends continuously to a bounded representation
$\sigma_x : C_0(K) \to \mathcal B(H)$ with
\[ \|\sigma_x\| \leq \|\pi\|_{cb}^2. \]
Let $\{f_\alpha \}_{\alpha \in \Delta }\subset A(K)$ be a contractive
approximate identity for $C_0(K)$.  For every $u\in A(G)$, we have
\[ \lim_{\alpha} \sigma_x(f_\alpha)\pi(u)
= \lim_{\alpha} \sigma_x(f_\alpha(L_{x^{-1}}u_x))
= \sigma_x(L_{x^{-1}}u_x) = \pi(u_x). \]
This, together with the fact that $\pi$ is non-degenerate, implies
that $E_x$, the limit of the net $\{\sigma_x(f_\alpha)\}$ in the strong operator
topology of $\mathcal B(H)$, exists. Moreover, for every $x\in I$, $E_x$ is an
idempotent and for every $x\neq y$,
\begin{align}\label{Eq:open subgroup-0}
E_xE_y=0.
\end{align}
Now let $A = \overline{\pi(A(G))} \subseteq \mathcal B(H)$ . Since $A$
is the closure of the range of a completely bounded homomorphism of the operator amenable
Banach algebra $A(G)$, it is itself operator amenable. Hence, by combining
\cite[Proposition 7.4.12(ii)]{BL} and \cite[Corollary 4.8]{G}, it follows that
\[ A''=\overline{A}^{s.o.t}, \]
where $\overline{A}^{s.o.t}$ refers to the closure of $A$ in the strong
operator topology of $\mathcal B(H)$. In particular, $A''$ is commutative and
for every $x\in I$, \[ E_x\in A''. \]
Therefore, again by combining \cite[Proposition 7.4.12(ii)]{BL} and \cite[Lemma 4.4]{G},
there is an invertible operator $T\in \mathcal B(H)$ such that
$T^{-1}E_xT$ is an orthogonal projection for all $x\in I$.  Let
\[ P_x = T^{-1}E_xT \quad\text{and}\quad
\rho_x(\cdot) = T^{-1}\sigma_x(\cdot) T \qquad (x\in I). \]
For every $x,y \in I$ with $x\neq y$, it follows from
(\ref{Eq:open subgroup-0}) that $P_xP_y =0.$
This implies that, for every $u,v\in C_0(K)$,
\begin{align}\label{Eq:open subgroup-1}
\rho_x(u)^*\rho_y(v)=\rho_x(u)^*P_x^*P_y\rho_y(v)=0,
\end{align}
and
\begin{align}\label{Eq:open subgroup-2}
\rho_x(u)\rho_y(v)^*=\rho_x(u)P_xP_y^*\rho_y(v)^*=0.
\end{align}
Now let $u\in A(G)$. For every $x,y \in I$ with $x\neq y$, it follows from
(\ref{Eq:open subgroup-1}) that
\[ \big( T\pi(u_x)T^{-1} \big)^*
\big( T\pi(u_y)T^{-1} \big)
= \rho_x(L_{x^{-1}}u_x)^*\rho_y(L_{y^{-1}}u_y) = 0. \]
Similarly (\ref{Eq:open subgroup-2}) implies that
\[ \big( T\pi(u_x)T^{-1} \big) \big(T\pi(u_y)T^{-1}\big)^*=0. \]
Hence $\{(T^{-1}\pi(u_x)T)^*(T^{-1}\pi(u_x)T) \}_{x\in I}$
is a family of pairwise orthogonal positive operators on $H$, and therefore
\begin{align*}
\|T\pi(u)T^{-1}\|^2 &=
\|(T\pi(u)T^{-1})^*(T\pi(u)T^{-1})\|
= \Big\|\sum_{x\in I} (T\pi(u_x)T^{-1})^*(T\pi(u_x)T^{-1})\Big\| \\
&= \sup_{x\in I} \| (T\pi(u_x)T^{-1})^*(T\pi(u_x)T^{-1})\| \leq
\|T\|^2 \|T^{-1}\|^2 \sup_{x\in I} \|\pi(u_x)\|^2 \\
&\leq   \|T\|^2\|T^{-1}\|^2 \|\pi\|_{cb}^4 \|u\|^2_\infty.
\end{align*}
Thus $T\pi(\cdot) T^{-1}$ extends continuously to a bounded homomorphism
of $C_0(G)$ into $\mathcal B(H)$, and hence, so does $\pi$.  The result now follows.
\end{proof}

One class of amenable locally compact groups having open SIN subgroups
are the amenable, totally disconnected groups. In particular, if we assume further
that they are non-compact, then they give us examples of non-SIN groups
satisfying the assumption of Theorem~\ref{T:cb rep-amen open SIN sub}. 
These groups, for instance, include Fell groups (see \cite[Section 2]{RS}
for more details).

\appendix
\section{Duality for closed operators}\label{sec:app1}

In this appendix, we will prove the following result.

\begin{proposition}\label{prop:aptwo}
Let $\G$ be a locally compact quantum group, and let $x,y\in L^\infty(\G)$
be such that $\ip{x}{\omega^\sharp} = \overline{\ip{y}{\omega}}$ for all
$\omega\in L^1_\sharp(\G)$.  Then $x\in D(S)$ with $S(x)^* = y$.
\end{proposition}

Firstly, note that as $S=R\tau_{-i/2}$, we have that $\omega\in L^1_\sharp(\G)$
if and only if there is $\omega'\in L^1(\G)$ with
\[ \ip{x}{\omega'} = \ip{\tau_{-i/2}(x)}{\omega^*} \qquad
(x\in D(\tau_{-i/2})). \]
Indeed, if this holds, then for $x\in D(S)$, we have that $\ip{x}{\omega'\circ R}
= \ip{\tau_{-i/2}(R(x))}{\omega^*} = \overline{ \ip{S(x)^*}{\omega} }$, so
$\omega^\sharp = \omega'\circ R$.  The converse follows similarly.

Set $D = L^1_\sharp(\G)^* = \{ \omega^* : \omega\in L^1_\sharp(\G) \}$
and for $\omega\in D$ let $\omega' = (\omega^*)^\sharp$.
It follows that the proposition above is equivalent to the following.

\begin{proposition}\label{prop:apone}
Let $\G$ be a locally compact quantum group, and let $x,y\in L^\infty(\G)$
be such that $\ip{x}{\omega'} = \ip{y}{\omega}$ for all
$\omega\in D$.  Then $x\in D(\tau_{-i/2})$ with $y = \tau_{-i/2}(x)$.
\end{proposition}

Perhaps the ``standard'' proof of the proposition would be to ``smear''
$x$ and $y$ by the one-parameter group $(\tau_t)$; compare
\cite[Proposition~4.22]{kus1} or the proof of \cite[Proposition~5.26]{kv},
for example.  Instead, we wish to indicate how to prove the result by
using closed operators.

Given a topological vector space $E$ and a subspace $D(T)\subseteq E$, a linear
map $T:D(T)\rightarrow E$ is \emph{closed and densely defined} if $D(T)$ is
dense in $E$, and the graph $\mc G(T) = \{ (x,T(x)) : x\in D(T) \}$ is closed
in $E\times E$.  Define also $\mc G'(T) = \{ (T(x),x) : x\in E \}$, which
is closed if and only if $\mc G(T)$ is.
Let $E$ be a Banach space-- we shall consider the case
when $E$ has the norm topology, and when $E^*$ has the weak$^*$-topology.

Let $T:D(T)\rightarrow E$ be a linear map.  Set $D(T^*) =
\{ \mu\in E^* : \exists\lambda\in E^*, \ \ip{\lambda}{x} = \ip{\mu}{T(x)}
\ (x\in D(T)) \}$, and with a slight abuse of notation, set
$\mc G(T^*)$ to be the collection of such $(\mu,\lambda)$.  Notice that:
\begin{itemize}
\item $\mc G'(-T^*) = \mc G(T)^\perp$, and so $\mc G(T^*)$ is weak$^*$-closed;
\item $D(T)$ is dense if and only if $(\lambda,0)\in \mc G(T)^\perp
\implies \lambda=0$, which is equivalent to $\mc G(T^*)$ being the graph
of an operator (that is, the choice of $\lambda$ being unique);
\item $\overline{ \mc G(T) } = {}^\perp( \mc G(T)^\perp )
= {}^\perp \mc G'(-T^*)$ and so $D(T^*)$ is weak$^*$-dense, because
$\mc G(T)$ is the graph of an operator.
\end{itemize}
It follows that $T^*$ is a closed, densely defined operator, for the
weak$^*$-topology.  We can similarly reverse the argument, and start
with $T^*$, and define a closed, densely defined operator $T$ by setting
$D(T) = \{ x\in E : \exists y\in E, \ \ip{\mu}{y} = \ip{T^*(\mu)}{x}
\ (\mu\in D(T^*)) \}$ and $T(x) = y$.  Furthermore, these two constructions
are mutual inverses, so if we start with $T^*$, form $T$, and then form
the adjoint of $T$, we recover $T^*$.  We are not aware of an entirely
satisfactory reference for this construction, but see \cite[Section~5.5,
Chapter~III]{kato}.

\begin{proof}[Proof of Proposition~\ref{prop:apone}]
We apply the preceding discussion to $\tau_{-i/2}$ acting on $L^\infty(\G)$.
As $\tau_{-i/2}$ is the analytic generator of the
$\sigma$-weakly continuous group $(\tau_t)$, it follows from \cite{cz}
that $\tau_{-i/2}$ is $\sigma$-weakly closed and densely defined.
Let $\tau_{*,-i/2}$ be the pre-adjoint, which from the above is hence
a (norm) closed, densely defined operator on $L^1(\G)$.  It follows
that $D = D(\tau_{*,-i/2})$ and that $\omega' = \tau_{*,-i/2}(\omega)$
for $\omega\in D$.  The hypothesis of the proposition now simply states
that $x\in D((\tau_{*,-i/2})^*)$ and $y=(\tau_{*,-i/2})^*(x)$.  The claim
now follows from the fact that $(\tau_{*,-i/2})^* = \tau_{i/2}$.
\end{proof}

We remark that, for example, this argument gives a very easy proof that
$L^1_\sharp(\G)$ is dense in $L^1(\G)$, without the need for a ``smearing''
argument (compare with the discussion before Proposition~3.1 in \cite{kus}).

\pagebreak

\small

\noindent
Michael Brannan\\
Department of Mathematics and Statistics,\\
Queen's University\\
Jeffery Hall, University Avenue, Kingston, \\
Ontario, Canada K7L 3N6.\\
Email: \texttt{mbrannan@mast.queensu.ca}

\medskip

\noindent
Matthew Daws\\
School of Mathematics,\\
University of Leeds,\\
LEEDS LS2 9JT\\
United Kingdom\\
Email: \texttt{matt.daws@cantab.net}

\medskip

\noindent
Ebrahim Samei\\
Department of Mathematics and Statistics\\
McLean Hall\\
University of Saskatchewan\\
106 Wiggins Road, Saskatoon,\\
Saskatchewan, Canada S7N 5E6\\
Email: \texttt{samei@math.usask.ca}


\begin{thebibliography}{9}
\small

\bibitem{av} D. Avitzour,
   ``Free products of {$C\sp{\ast} $}-algebras'',
   \emph{Trans. Amer. Math. Soc.} 271 (1982) 423--435.

\bibitem{bs2} S. Baaj, G. Skandalis,
   ``Unitaires multiplicatifs et dualit\'e pour les produits crois\'es de
   $C^*$-alg\`ebres'', 
   \emph{Ann. Sci. \'Ecole Norm. Sup.} 26 (1993) 425--488. 

\bibitem{basp} T. Banica, R. Speicher,
   ``Liberation of orthogonal Lie groups'',
   \emph{Adv. Math.} 222 (2009) 1461--1501.

\bibitem{bt} E. B{\'e}dos, L. Tuset,
   ``Amenability and co-amenability for locally compact quantum groups'',
   \emph{Internat. J. Math.} 14 (2003) 865--884.

\bibitem{BL} D. Blecher, C. Le Merdy, ``Operator algebras and their modules:
   an operator space approach'',
   \emph{London Mathematical Society Monographs, New Series, 30}
   (Oxford University Press, Oxford, 2004).

\bibitem{bsm} D.\,P. Blecher, R.\,R. Smith,
   ``The dual of the {H}aagerup tensor product'',
   \emph{J. London Math. Soc. (2)} 45 (1992) 126--144.

\bibitem{bs} M. Brannan, E. Samei,
   ``The similarity problem for {F}ourier algebras and
              corepresentations of group von {N}eumann algebras'',
   \emph{J. Funct. Anal.} 259 (2010) 2073--2097.

\bibitem{cs} Y. Choi, E. Samei,
   ``Quotients of the Fourier algebra, and representations
   that are not completely bounded'',
   preprint, see arXiv:1104.2953 [math.FA]

\bibitem{ch} E. Christensen,
   ``On nonselfadjoint representations of {$C^{\ast} $}-algebras'',
   \emph{Amer. J. Math.} 103 (1981) 817--833.

\bibitem{cz}  I. Cior{\u{a}}nescu, L. Zsid{\'o},
   ``Analytic generators for one-parameter groups'',
   \emph{T\^ohoku Math. J.} 28 (1976) 327--362.

\bibitem{coha} M. Cowling, U. Haagerup,
   ``Completely bounded multipliers of the Fourier algebra of a simple Lie group
   of real rank one'',
   \emph{Invent. Math.} 96 (1989) 507--549.

\bibitem{day} M. Day, ``Means for the bounded functions and ergodicity of the
   bounded representations of semi-groups'',
   \emph{Trans. Amer. Math. Soc.} 69 (1950) 276--291.

\bibitem{dm} M. Daws,
   ``Multipliers, Self-Induced and Dual Banach Algebras'',
   \emph{Dissertationes Math.} 470 (2010) 62pp.

\bibitem{dm2} M. Daws, ``Multipliers of locally compact quantum groups
   via Hilbert C$^*$-modules'', \emph{J. London Math. Soc.}
   84 (2011) 385--407.

\bibitem{dh} J. de Canniere, U. Haagerup,
   ``Multipliers of the Fourier algebras of some simple Lie Groups and
   their discrete subgroups'',
   \emph{Amer. Journ. Math.} 107 (1985) 455--500.

\bibitem{dix} J. Dixmier, ``Les moyennes invariantes dans les semi-groupes
   et leurs applications'',
   \emph{Acta Sci. Math.} (Szeged) 12 (1950) 213--227.

\bibitem{ER} E.\,G. Effros, Z.-J. Ruan,
   ``Operator spaces'',\emph{London Mathematical Society Monographs, New
   Series, 23}
   (Oxford University Press, New York, 2000).

\bibitem{er2} E.\,G. Effros, Z.-J. Ruan,
   ``Operator space tensor products and {H}opf convolution algebras'',
   \emph{J. Operator Theory} 50 (2003) 131--156.

\bibitem{eymard} P. Eymard,
   ``L'alg\`ebre de {F}ourier d'un groupe localement compact'',
   \emph{Bull. Soc. Math. France} 92 (1964) 181--236.

\bibitem{G} J. A. Gifford,
   ``Operator algebras with a reduction property'',
   \emph{J. Austral. Math. Soc.} 80 (2006) 297--315.

\bibitem{ju} M. Junge,
   ``Embedding of the operator space $OH$ and the logarithmic
   `little Grothendieck inequality'{''}
   \emph{Invent. Math.} 161 (2005) 225--286.

\bibitem{jnr} M. Junge, M. Neufang, Z.-J. Ruan,
   ``A representation theorem for locally compact quantum groups'',
   \emph{Internat. J. Math}. 20 (2009), 377--400.

\bibitem{kato} T. Kato
   ``Perturbation theory for linear operators.''
   Second edition. Grundlehren der Mathematischen Wissenschaften, Band 132.
   (Springer-Verlag, Berlin-New York, 1976).

\bibitem{kus} J. Kustermans,
   ``Locally compact quantum groups in the universal setting'',
   \emph{Internat. J. Math.} 12 (2001) 289--338.

\bibitem{kus2} J. Kustermans,
   ``Induced corepresentations of locally compact quantum groups'',
   \emph{J. Funct. Anal.} 194 (2002) 410--459.

\bibitem{kus1} J. Kustermans,
   ``Locally compact quantum groups'' in
   \emph{Quantum independent increment processes. {I}},
   {Lecture Notes in Math.} {1865}, pp.\ {99--180}
   (Springer, Berlin, 2005).

\bibitem{kvvn} J. Kustermans, S. Vaes,
   ``Locally compact quantum groups in the von {N}eumann algebraic setting'',
   \emph{Math. Scand.} 92 (2003) 68--92.

\bibitem{kv} J. Kustermans, S. Vaes,
   ``Locally compact quantum groups'',
   \emph{Ann. Sci. \'Ecole Norm. Sup.} 33 (2000) 837--934.

\bibitem{mo} N. Monod, N. Ozawa,
   ``The Dixmier problem, lamplighters and Burnside groups'',
   \emph{J. Funct. Anal.} 258 (2010) 255--259. 

\bibitem{nisp} A. Nica and R. Speicher, R., ``Lectures on the combinatorics of free probability''  \emph{LMS Lecture Notes Series 335} (Cambridge University Press,
2006).

\bibitem{pis} G. Pisier,
   ``Are unitarizable groups amenable?''
   in ``Infinite groups: geometric, combinatorial and dynamical aspects'',
   323--362, \emph{Progr. Math.} 248 (Birkh\"auser, Basel, 2005).

\bibitem{ps} T. Pytlik, R. Szwarc,
   ``An analytic family of uniformly bounded representations of free groups'',
   \emph{Acta Math.} 157 (1986) 287--309.

\bibitem{rx} \'E. Ricard, Q. Xu,
   ``Khintchine type inequalities for reduced free products and applications''
   \emph{J. Reine Angew. Math.} 599 (2006) 27--59. 

\bibitem{Ruan} Z.-J.Ruan,
   ``The operator amenability of $A(G)$'',
   \emph{Amer. J. Math.} 117 (1995), 1449--1474.

\bibitem{ru} W. Rudin,
   ``Functional Analysis'',
   (McGraw-Hill Publishing Co., New York, 1973).

\bibitem{RS} V. Runde, N. Spronk,
   ``Operator amenability of Fourier-Stieltjes algebras. II'',
   \emph{Bull. Lond. Math. Soc.}  39  (2007)  194--202.

\bibitem{tak} M. Takesaki,
   ``Theory of operator algebras. I.'',
   \emph{Encyclopaedia of Mathematical Sciences, 124.
   Operator Algebras and Non-commutative Geometry, 5.}
   (Springer-Verlag, Berlin, 2002).

\bibitem{tak2} M. Takesaki,
   ``Theory of operator algebras. II.''
   \emph{Encyclopaedia of Mathematical Sciences, 125.
   Operator Algebras and Non-commutative Geometry, 6.}
   (Springer-Verlag, Berlin, 2003).

\bibitem{tt2} M. Takesaki, N. Tatsuuma,
   ``Duality and subgroups. {II}'',
   \emph{J. Functional Analysis} 11 (1972) 184--190.

\bibitem{tay} K.\,F. Taylor,
   ``The type structure of the regular representation of a locally compact group'',
   \emph{Math. Ann.} 222 (1976) 211--224.

\bibitem{tt} T. Timmermann,
   ``An invitation to quantum groups and duality'',
   (European Mathematical Society, Z\"urich, 2008).
   
\bibitem{tom} R. Tomatsu,
   ``Amenable discrete quantum groups'',
   \emph{J. Math. Soc. Japan} 58 (2006) 949--964.

\bibitem{vphd} S. Vaes, ``Locally compact quantum groups'',
   PhD. thesis, Katholieke Universiteit Leuven, 2001.
   Available from http://wis.kuleuven.be/analyse/stefaan/

\bibitem{vv} S. Vaes, A. Van Daele,
   ``Hopf {$C\sp *$}-algebras'',
   \emph{Proc. London Math. Soc.} 82 (2001) 337--384.

\bibitem{Voi} D. Voiculescu,
   ``A strengthened asymptotic freeness result for random matrices
   with applications to free entropy''
   \emph{Internat. Math. Res. Notices} 1 (1998) 41--63.

\bibitem{vo} D. Voiculescu,
   ``Symmetries of some reduced free product C$^*$-algebras''
   in \emph{Operator algebras and their connections with topology and
  ergodic theory} (Bu\c{s}teni, 1983) 556--588.
   \emph{Lecture Notes in Math.} 1132 (Springer, Berlin, 1985).

\bibitem{wa2} S. Wang,
   ``Quantum symmetry groups of finite spaces'',
   \emph{Comm. Math. Phys.} 195 (1998) 195--211.

\bibitem{wa1} S. Wang,
   ``Free products of compact quantum groups'',
   \emph{Comm. Math. Phys.} 167 (1995) 671--692.

\bibitem{woro2} S.\,L. Woronowicz,
   ``Compact quantum groups''
   in \emph{Sym\'etries quantiques (Les Houches, 1995)}
   pp. 845--884 (North-Holland, Amsterdam, 1998).

\bibitem{woro} S.\,L. Woronowicz,
   ``From multiplicative unitaries to quantum groups'',
   \emph{Internat. J. Math.} 7 (1996) 127--149. 

\bibitem{woro1} S.\,L. Woronowicz,
   ``Compact matrix pseudogroups'',
   \emph{Comm. Math. Phys.} 111 (1987) 613--665.


\end{thebibliography}
\end{document}